\documentclass[12pt]{article}
\usepackage{amsmath, amsfonts, amssymb, amsthm, a4wide}
\usepackage[arrow]{xy}
\numberwithin{equation}{section}
\newcommand{\pa}[1]{\overline{#1}}
\newcommand{\gl}{\mathfrak{gl}}
\title{Parabolic presentations of the super Yangian $Y(\gl_{M|N})$}
\author{Yung-Ning Peng\\
\normalsize{yp9s@virginia.edu}\\
\normalsize{Department of Mathematics}\\
\normalsize{University of Virginia}}\begin{scriptsize}\begin{footnotesize}\end{footnotesize}\end{scriptsize}
\begin{document}
\date{}
\maketitle
\newtheorem{lemma}{Lemma}[section]
\newtheorem{corollary}[lemma]{Corollary}
\newtheorem{proposition}[lemma]{Proposition}
\newtheorem{definition}{Definition}[section]
\newtheorem{theorem}{Theorem}
\newtheorem{remark}{Remark}[section]

\maketitle
\abstract Associated to a composition of M and a composition of N, a new presentation of the super Yangian of the general linear Lie superalgebra $Y(\gl_{M|N})$ is obtained.

\tableofcontents

\section{Introduction}
For each simple finite-dimensional Lie algebra $\mathfrak{g}$ over $\mathbb{C}$, the associated Yangian $Y(\mathfrak{g})$ was defined by Drinfeld in \cite{D1} as a deformation of the universal enveloping algebra $U(\mathfrak{g}[x])$ for the polynomial current Lie algebra $\mathfrak{g}[x]$. The Yangians form a family of quantum groups which give rise to rational solutions of the Yang-Baxter equation originating from statistical mechanics; see \cite{CP}. A Yangian admits PBW basis, triangular decomposition and Hopf algebra structure. The Yangian $Y(\gl_N)$  of the reductive Lie algebra $\gl_N$ was earlier considered in \cite{TF}. It is an associative algebra whose defining relations can be written in a specific matrix form, which is called the RTT relation; see e.g. \cite{FRT} and \cite{MNO}. The structures and representation theory of $Y(\gl_N)$ have been studied by many people; see e.g. \cite{KRS}, \cite{Ta}, \cite{MNO},  and \cite{Mo}. In \cite{D2}, Drinfeld gave a new presentation for Yangians and it in particular can be used to define the analog of the Cartan subalgebra and the Borel subalgebra in $Y(\gl_N)$.

In \cite{BK1}, Brundan and Kleshchev found a parabolic presentation for $Y(\gl_N)$ associated to each composition $\lambda$ of N. Roughly speaking, the new presentation corresponds to a block matrix decomposition of $\gl_N$ of shape $\lambda$. In the special case when $\lambda=(1,1,\ldots,1)$, the corresponding parabolic presentation is just a variation of Drinfeld's; see \cite[Remark 5.12]{BK1}. On the other extreme case when $\lambda=(N)$, the corresponding parabolic presentation is exactly the original RTT presentation. The parabolic presentation allows Brundan and Kleshchev to further define the standard Levi and parabolic subalgebras of $Y(\gl_N)$, and thus to obtain a Levi decomposition of $Y(\gl_N)$. The parabolic presentations have played a crucial role in their subsequent work \cite{BK2}, in which they derived generators and relations for the finite $W$-algebras.

The main goal of this article is to obtain the superalgebra generalization of the parabolic presentations of \cite{BK1} for the super Yangian $Y(\gl_{M|N})$. The super Yangian of the general linear Lie superalgebra $Y(\gl_{M|N})$ was introduced by Nazarov in \cite{Na}, and it shares many properties with the usual Yangian, such as the PBW theorem, the RTT relation and the Hopf algebra structure. The results of this article will be used in a sequel on the connection between $Y(\gl_{M|N})$ and the super $W$-algebras.

Let $\lambda$ be a composition of $M$ and $\nu$ be a composition of $N$. We first define some distinguished elements in $Y(\gl_{M|N})$, denoted by D's, E's and F's, by Gauss decomposition and quasideterminants. We show that these new elements form a set of generators for $Y(\gl_{M|N})$. The next step is to find the relations among the new generators, where the signs arising from the $\mathbb{Z}_2$-grading are involved here. However, since the $(\lambda |\nu)$-block decomposition of $\gl_{M|N}$ respects the $\mathbb{Z}_2$-grading of the superalgebra, the signs in the relations are determined by the block positions only. It is known (cf. \cite{BK1}) that if the elements are from two different blocks and the blocks are not ``close", then they commute. This phenomenon remains to be true in our super Yangian setting and it dramatically reduces the number of the nontrivial relations. Hence we only have to focus on the commutation relations of the elements in the same block or when their block-positions are ``close". Let $m$ be the number of parts of $\lambda$ and $n$ be the number of parts of $\nu$. Then the first new non-trivial case will be $m=n=1$, and the new ones will be $m=2$, $n=1$ and $m=1$, $n=2$ (see Section 4). In these special cases, we determined various relations among D's, E's  and F's by direct computation.

Next, we make use of the shift map $\psi_k$ and the swap map $\zeta_{M|N}$ between super Yangians (see Section 4). These maps allow us to transfer the relations in the special cases with $m+n\leq 3$ to relations in $Y(\gl_{M|N})$ in the setting of general compositions $\lambda$ and $\nu$. Finally we show that we have found enough relations for our new presentation. As a consequence, we obtain the PBW bases for several distinguished subalgebras of $Y(\gl_{M|N})$.

The parabolic presentation in the extreme case when all parts of $\lambda$ and $\nu$ are 1 was found by Gow in \cite{Go}, who used the presentation to define the super Yangian of the special linear superalgebra $Y(\mathfrak{sl}_{N|N})$ which was missing in the literature and to determine the generators of the center of $Y(\gl_{M|N})$. However, there are non-trivial relations that can not be observed in this special case and nevertheless play an important role in our paper (see Remark 7.1 below).

We organize this article in the following manner. In Section 2, we recall the definition and some properties of $Y(\gl_{M|N})$. In Section 3, we introduce the generating elements in our parabolic presentation by means of Gauss decomposition. In Section 4, we define some maps between super Yangians in order to reduce the general case to special cases when $m+n\leq 3$, and Section 5 and 6 are devoted to these special cases. Our main theorem in the general case is formulated in Section 7 and its proof is completed in Section 8.

\section{Properties of the super Yangian $Y(\gl_{M|N})$}
Most of the theorems and lemmas in Sections 2 to Section 4 are generalizations of the counterparts for $Y(\gl_N)$ in \cite{MNO} or \cite{BK1}.

The super Yangian $Y (\gl_{M|N})$, which was introduced in \cite{Na}, is the associative $\mathbb{Z}_2$-graded algebra (i.e., superalgebra) over $\mathbb{C}$ with generators
\[
\left\lbrace t_{ij}^{(r)}\,| \; 1\le i,j \le M+N; r\ge 0\right\rbrace,
\] where $t_{ij}^{(0)}:=\delta_{ij}$
and defining relations
\begin{equation}\label{yangiandef}
[t_{ij}^{(r)}, t_{hk}^{(s)}] = (-1)^{\pa{i}\,\pa{j} + \pa{i}\,\pa{h} + \pa{j}\,\pa{h}}
\sum_{t=0}^{\mathrm{min}(r,s) -1} \Big( t_{hj}^{(t)} t_{ik}^{(r+s-1-t)} - t_{hj}^{(r+s-1-t)}t_{ik}^{(t)} \Big),
\end{equation}
where $\pa{i} = 0$ if $i\le M$, $\pa{i}=1$ if $i \ge M+1$, and the bracket is understood as a supercommutator. For $r>0$, the element $t_{ij}^{(r)}$ is defined to be an odd element if $\pa{i}+\pa{j}=\pa{1}$ and an even element if $\pa{i}+\pa{j}=\pa{0}$.

\begin{remark} When N=0, the super Yangian $Y(\gl_{M|0})$ is naturally isomorphic to the usual Yangian $Y(\gl_M)$; when M=0, the super Yangian $Y(\gl_{0|N})$ is also isomorphic to the usual Yangian $Y(\gl_N)$ by the map $\zeta_{0|N}$, see Section 4.
\end{remark}

We define the formal power series to be the generating series (with non-positive powers of a variable $u$) of the generators:
\[
t_{ij} (u) = \delta_{ij} + t_{ij}^{(1)} u^{-1} + t_{ij}^{(2)}u^{-2} + t_{ij}^{(3)}u^{-3}+\ldots.
\]
Also define
\[
T(u):=\sum_{i,j=1}^{M+N}t_{ij}(u)\otimes E_{ij}(-1)^{\pa{j}(\pa{i}+1)}\in Y(\gl_{M|N})[[u^{-1}]]\otimes \text{End } \mathbb{C}^{M|N},
\]
where $E_{ij}$ is the standard elementary matrix. The extra sign ensures that the product of matrices can be calculated in the usual manner. We may also think $T(u)$ as an element in $Mat_{M+N}\Big(Y(\gl_{M|N})[[u^{-1}]]\Big)$, the set of (M+N)$\times$(M+N) matrices with entries in $Y(\gl_{M|N})[[u^{-1}]]$.

We may also define the super Yangian $Y(\gl_{M|N})$ by the \textit{RTT relation}:
\begin{equation}\label{RTT}
R(u-v)T_1(u)T_2(v)=T_2(v)T_1(u)R(u-v),
\end{equation}
where \[T_1(u)=T(u)\otimes Id_{M+N},\quad \quad T_2(v)=Id_{M+N}\otimes T(v), \quad R(u-v)=1-\dfrac{P_{12}}{(u-v)},\]
\[ \text{and}\qquad P_{12}=\sum_{i,j=1}^{M+N}(-1)^{\pa{j}}E_{ij}\otimes E_{ji} \quad \text{is the permutation matrix.}\]
The equality is in $Mat_{M+N}\otimes Mat_{M+N}\otimes Y(\gl_{M|N})((u^{-1},v^{-1}))$, which means the localization of $Mat_{M+N}\otimes Mat_{M+N}\otimes Y(\gl_{M|N})[[u^{-1},v^{-1}]]$ at the multiplicative set consisting of the non-zero elements of $\mathbb{C}[[u^{-1},v^{-1}]]$.

\begin{remark}
Note that we have $(u-v)^{-1}$ in the matrix $R(u-v)$. Hence we have to replace $Y(\gl_{M|N})[[u^{-1},v^{-1}]]$ by a certain extension containing $(u-v)^{-1}$.
\end{remark}

Equating the coefficients of $E_{ij}\otimes E_{hk}$ on both sides of (2.2), we have the following equivalent defining relations in terms of the generating series:
\begin{equation}\label{ydefs}
[t_{ij}(u),t_{hk}(v)]=\frac{(-1)^{\pa{i}\,\pa{j}+\pa{i}\,\pa{h}+\pa{j}\,\pa{h}}}{(u-v)}
\Big(t_{hj}(u)t_{ik}(v)-t_{hj}(v)t_{ik}(u)\Big).
\end{equation}
Note that the matrix $T(u)$ is invertible, hence one may define the entries of its inverse by
\[
\big(T(u)\big)^{-1}:=\big(t_{ij}^{\prime}(u)\big)_{i,j=1}^{M+N}.
\]
Multiplying $T_2(v)^{-1}$ on both sides of (2.2) and use the same method getting (2.3), we have yet another relation:
\begin{equation}\label{usefull}
[t_{ij}(u),t'_{hk}(v)]
=\frac{(-1)^{\pa{i}\,\pa{j}+\pa{i}\,\pa{h}+\pa{j}\,\pa{h}}}{(u-v)}\Big(\delta_{h,j}\sum_{l=1}^{M+N}
t_{il}(u)t'_{lk}(v)-\delta_{i,k}\sum_{l=1}^{M+N}t'_{hl}(v)t_{lj}(u)\Big)  .
\end{equation}
As an easy consequence of (\ref{usefull}), we know that for all $r$ and $s$, if $i\neq k$ and $j\neq h$, then $t_{ij}^{(r)}$ and $t_{hk}^{\prime(s)}$ supercommute.
The following is the PBW basis theorem for $Y(\gl_{M|N})$.
\begin{proposition}\cite[Theorem 1]{Go}\label{PBWSY}
The set of all monomials in the elements
\[
\left\lbrace t_{ij}^{(r)}| 1\leq i,j\leq M+N, r\geq 1\right\rbrace
\] taken in some fixed order (containing no second or higher order powers of the odd generators) forms a basis for $Y(\gl_{M|N})$.
\end{proposition}
We have the $loop$ $filtration$ on $Y(\gl_{M|N})$
\begin{equation}\label{filt2}\notag
L_0 Y(\gl_{M|N}) \subseteq L_1 Y(\gl_{M|N}) \subseteq L_2 Y(\gl_{M|N}) \subseteq \cdots
\end{equation}
defined by setting deg $t_{ij}^{(r)}=r-1$ for each $r\geq 1$ and $L_kY(\gl_{M|N})$ is the span of all monomials of the form $t_{i_1j_1}^{(r_1)}t_{i_2j_2}^{(r_2)}\cdots t_{i_sj_s}^{(r_s)}$ with total degree $\sum_{i=1}^s(r_i-1)\leq k$. We denote the associated graded algebra by $gr^LY(\gl_{M|N}).$

Let $\gl_{M|N}[t]$ denote the loop superalgebra $\gl_{M|N}\otimes \mathbb{C}[t]$ with the standard basis $\lbrace E_{ij}t^r \,|\, 1\leq i,j\leq M+N, r\geq 0\rbrace$ and $U(\gl_{M|N}[t])$ denote its universal enveloping algebra. By the PBW theorem for $Y(\gl_{M|N})$, we have the following corollary.

\begin{corollary}\cite[Corollary 1]{Go}\label{Yloop}
The graded algebra $gr^LY(\gl_{M|N})$ is isomorphic to the universal enveloping algebra $U(\gl_{M|N}[t])$ by the map
\begin{center}
$gr^LY(\gl_{M|N})\rightarrow U(\gl_{M|N}[t])$\\
$gr^L_{r-1}t_{ij}^{(r)}\mapsto (-1)^{\pa{i}}E_{ij}t^{r-1}.$
\end{center}
\end{corollary}
\section{Gauss decomposition and quasideterminants}\label{Gauss}
Let $\lambda$ be a composition of $M$ and $\nu$ be a composition of $N$. In the remaining part of this article, for notational reason, we set
\[
\mu_i=\lambda_i \qquad\text{and}\qquad \mu_{m+j}=\nu_j \quad\text{for all}\quad 1\leq i\leq m,\;\; 1\leq j\leq n,
\]
and $\mu=(\mu_1,\mu_2,\ldots,\mu_{m}\,|\,\mu_{m+1},\mu_{m+2},\ldots,\mu_{m+n})$ denotes the composition of ($M|N$).

By definition, the leading minors of the matrix $T(u)$ are invertible. Then it possesses a $Gauss$ $decomposition$ (cf.~\cite{GR})
\begin{equation*}
T(u) = F(u) D(u) E(u)
\end{equation*}
for unique {\em block matrices} $D(u)$, $E(u)$ and $F(u)$ of the form
$$
D(u) = \left(
\begin{array}{cccc}
D_{1}(u) & 0&\cdots&0\\
0 & D_{2}(u) &\cdots&0\\
\vdots&\vdots&\ddots&\vdots\\
0&0 &\cdots&D_{m+n}(u)
\end{array}
\right),
$$

$$
E(u) =
\left(
\begin{array}{cccc}
I_{\mu_1} & E_{1,2}(u) &\cdots&E_{1,m+n}(u)\\
0 & I_{\mu_2} &\cdots&E_{2,m+n}(u)\\
\vdots&\vdots&\ddots&\vdots\\
0&0 &\cdots&I_{\mu_{m+n}}
\end{array}
\right),\:
$$

$$
F(u) = \left(
\begin{array}{cccc}
I_{\mu_1} & 0 &\cdots&0\\
F_{2,1}(u) & I_{\mu_2} &\cdots&0\\
\vdots&\vdots&\ddots&\vdots\\
F_{m+n,1}(u)&F_{m+n,2}(u) &\cdots&I_{\mu_{m+n}}
\end{array}
\right),
$$
where
\begin{align}
D_a(u) &=\big(D_{a;i,j}(u)\big)_{1 \leq i,j \leq \mu_a},\\
E_{a,b}(u)&=\big(E_{a,b;i,j}(u)\big)_{1 \leq i \leq \mu_a, 1 \leq j \leq \mu_b},\\
F_{b,a}(u)&=\big(F_{b,a;i,j}(u)\big)_{1 \leq i \leq \mu_b, 1 \leq j \leq \mu_a},
\end{align}
are $\mu_a \times \mu_a$,
$\mu_a \times \mu_b$
and  $\mu_b \times\mu_a$ matrices, respectively, for all $1\le a\le m+n$ in (3.1)
and all $1\le a<b\le m+n$ in (3.2) and (3.3).

\begin{definition}
We call the indices $a,b$ the $block$ $positions$, and the indices $i,j$ the $entry$ $positions$.
\end{definition}
Also define the $\mu_a\times\mu_a$ matrix
$D_a^{\prime}(u)=\big(D_{a;i,j}^{\prime}(u)\big)_{1\leq i,j\leq \mu_a}$ by
\begin{equation*}
D_a^{\prime}(u):=\big(D_a(u)\big)^{-1}.
\end{equation*}
The entries of these matrices are expanded into power series
\begin{eqnarray*}
&D_{a;i,j}(u) =& \sum_{r \geq 0} D_{a;i,j}^{(r)} u^{-r},\\
&D_{a;i,j}^{\prime}(u) =& \sum_{r \geq 0}D^{\prime(r)}_{a;i,j} u^{-r},\\
&E_{a,b;i,j}(u) =& \sum_{r \geq 1} E_{a,b;i,j}^{(r)} u^{-r},\\
&F_{b,a;i,j}(u) =& \sum_{r \geq 1} F_{b,a;i,j}^{(r)} u^{-r}.
\end{eqnarray*}
Moreover, for $1\leq a\leq m+n-1$, we set
\begin{eqnarray*}
E_{a;i,j}(u) :=& E_{a,a+1;i,j}(u)=\sum_{r \geq 1} E_{a;i,j}^{(r)} u^{-r},\\
F_{a;i,j}(u) :=& F_{a+1,a;i,j}(u)=\sum_{r \geq 1} F_{a;i,j}^{(r)} u^{-r}.
\end{eqnarray*}

There are explicit descriptions of all these series in terms of \textit{quasideterminants} (cf.~\cite{GKLLRT}, \cite{GR}). To write them down, we introduce the following notation. Suppose that $A, B, C$ and $D$ are $a \times a$, $a \times b$, $b \times a$ and $b \times b$ matrices respectively with entries in some ring.
Assuming that the matrix $A$ is invertible, we define
\begin{equation*}
\left|
\begin{array}{cc}
A&B\\
C&
\hbox{\begin{tabular}{|c|}\hline$D$\\\hline\end{tabular}}
\end{array}
\right| := D - C A^{-1} B.
\end{equation*}
We write the matrix $T(u)$ in block form as
$$
T(u) = \left(
\begin{array}{lll}
{^\mu}T_{1,1}(u)&\cdots&{^\mu}T_{1,m+n}(u)\\
\vdots&\ddots&\cdots\\
{^\mu}T_{m+n,1}(u)&\cdots&{^\mu}T_{m+n,m+n}(u)\\
\end{array}
\right),
$$ where each ${^\mu}T_{a,b}(u)$ is a $\mu_a \times \mu_b$ matrix.

\begin{proposition}\cite{GR}\label{quasi} We have
\begin{align}\label{quasid}
&D_a(u) =
\left|
\begin{array}{cccc}
{^\mu}T_{1,1}(u) & \cdots & {^\mu}T_{1,a-1}(u)&{^\mu}T_{1,a}(u)\\
\vdots & \ddots &\vdots&\vdots\\
{^\mu}T_{a-1,1}(u)&\cdots&{^\mu}T_{a-1,a-1}(u)&{^\mu}T_{a-1,a}(u)\\
{^\mu}T_{a,1}(u) & \cdots & {^\mu}T_{a,a-1}(u)&
\hbox{\begin{tabular}{|c|}\hline${^\mu}T_{a,a}(u)$\\\hline\end{tabular}}
\end{array}
\right|,\\[4mm]
&E_{a,b}(u) =\label{quasie}
D^{\prime}_a(u)
\left|\begin{array}{cccc}
{^\mu}T_{1,1}(u) & \cdots &{^\mu}T_{1,a-1}(u)& {^\mu}T_{1,b}(u)\\
\vdots & \ddots &\vdots&\vdots\\
{^\mu}T_{a-1,1}(u) & \cdots & {^\mu}T_{a-1,a-1}(u)&{^\mu}T_{a-1,b}(u)\\
{^\mu}T_{a,1}(u) & \cdots & {^\mu}T_{a,a-1}(u)&
\hbox{\begin{tabular}{|c|}\hline${^\mu}T_{a,b}(u)$\\\hline\end{tabular}}
\end{array}
\right|,\\[4mm]
&F_{b,a}(u) =\label{quasif}
\left|
\begin{array}{cccc}
{^\mu}T_{1,1}(u) & \cdots &{^\mu}T_{1,a-1}(u)& {^\mu}T_{1,a}(u)\\
\vdots & \ddots &\vdots&\vdots\\
{^\mu}T_{a-1,1}(u) & \cdots & {^\mu}T_{a-1,a-1}(u)&{^\mu}T_{a-1,a}(u)\\
{^\mu}T_{b,1}(u) & \cdots & {^\mu}T_{b,a-1}(u)&
\hbox{\begin{tabular}{|c|}\hline${^\mu}T_{b,a}(u)$\\\hline\end{tabular}}
\end{array}
\right|D^{\prime}_a(u),
\end{align}
for all $1\leq a\leq m+n$ in (\ref{quasid}) and $1\leq a<b\leq m+n$ in (\ref{quasie}), (\ref{quasif}).
\end{proposition}

We denote the $(i,j)$-th entry of the $\mu_a\times\mu_b$ matrix $^\mu T_{a,b}(u)$ by $T_{a,b;i,j}(u)$ and denote the  coefficient of $u^{-r}$ in $T_{a,b;i,j}(u)$ by $T_{a,b;i,j}^{(r)}$. By Proposition 3.1, we immediately have
\begin{equation}
E_{b-1;i,j}^{(1)}=T_{b-1,b;i,j}^{(1)}, \qquad F_{b-1;i,j}^{(1)}=T_{b,b-1;i,j}^{(1)}, \qquad
\text{for all admissible} \quad b,i,j,
\end{equation}\label{Jacobi}
and\begin{equation}\label{Dtident}
D_{1;i,j}^{(r)}=T_{1,1;i,j}^{(r)}=t_{i,j}^{(r)},\quad \text{for all}\quad 1\leq i,j\leq \mu_1, r\geq 0.
\end{equation}
By induction, one may show that for each pair $a$, $b$ such that $1<a+1<b\leq m+n-1$ and $1\leq i\leq\mu_a$, $1 \leq j \leq \mu_b$, we have
\begin{equation}
E_{a,b;i,j}^{(r)} = (-1)^{\overline{b-1}}[E_{a,b-1;i,k}^{(r)}, E_{b-1;k,j}^{(1)}],
\qquad
F_{b,a;i,j}^{(r)} = (-1)^{\overline{b-1}}[F_{b-1;i,k}^{(1)}, F_{b-1,a;k,j}^{(r)}],\label{ter}
\end{equation} for any $1 \leq k \leq \mu_{b-1}$. Here, $\pa{a}:=0$ if $1\le a\le m$ and $\pa{a}:=1$ if $m+1\le a\le m+n$.

By multiplying out the matrix product $T(u)=F(u)D(u)E(u)$, we see that each $t_{ij}^{(r)}$ can be expressed as a sum of monomials in $D_{a;i,j}^{(r)}$, $E_{a,b;i,j}^{(r)}$ and $F_{b,a;i,j}^{(r)}$, appearing in certain order that all $F$'s before $D$'s and all $D$'s before $E$'s. By (\ref{ter}), it is enough to use $D_{a;i,j}^{(r)}$, $E_{a;i,j}^{(r)}$ and $ F_{a;i,j}^{(r)}$ only, rather than all $E$'s and $F$'s. We have proved the following theorem.

\begin{theorem}\label{gendef}
The super Yangian $Y(\gl_{M|N})$ is generated as an algebra by the following elements
\begin{align*}
&\big\lbrace D_{a;i,j}^{(r)}, D_{a;i,j}^{\prime(r)} \,|\, {1\leq a\leq m+n,\; 1\leq i,j\leq \mu_a,\; r\geq 0}\big\rbrace,\\
&\big\lbrace E_{a;i,j}^{(r)} \,|\, {1\leq a< m+n,\; 1\leq i\leq \mu_a, 1\leq j\leq\mu_{a+1},\; r\geq 1}\big\rbrace,\\
&\big\lbrace F_{a;i,j}^{(r)} \,|\, {1\leq a< m+n,\; 1\leq i\leq\mu_{a+1}, 1\leq j\leq \mu_a,\; r\geq 1}\big\rbrace.
\end{align*}
\end{theorem}

\section{Maps between super Yangians}

Our ultimate goal in this article is to find out the defining relations among the generating elements $\big\lbrace D_{a;i,j}^{(r)}, D_{a;i,j}^{\prime(r)}, E_{a;i,j}^{(r)}, F_{a;i,j}^{(r)}\big\rbrace$ in $Y(\gl_{M|N})$.
The strategy is to work out the special cases when $m$ and $n$ are either 1 or 2, which are relatively less complicated, and then to apply the maps in this section to obtain the relations in the general case.
\begin{proposition}
\begin{enumerate}
  \item [(1)]The map $\rho_{M|N}:Y(\gl_{M|N})\rightarrow Y(\gl_{N|M})$ defined by
  \[
  \rho_{M|N}\big(t_{ij}(u)\big)=t_{M+N+1-i,M+N+1-j}(-u)
  \]
   is an algebra isomorphism.
  \item [(2)]The map $\omega_{M|N}:Y(\gl_{M|N})\rightarrow Y(\gl_{M|N})$ defined by
  \[
  \omega_{M|N}\big(T(u)\big)=\big(T(-u)\big)^{-1}
  \]
   is an algebra automorphism.
  \item [(3)]For any $k\in\mathbb{Z}_{\ge0}$, the map $\psi_k:Y(\gl_{M|N})\rightarrow Y(\gl_{k+M|N})$ defined by
  \[
  \psi_k=\omega_{k+M|N}\circ\varphi_{M|N}\circ\omega_{M|N},
  \]
  where $\varphi_{M|N}:Y(\gl_{M|N})\rightarrow Y(\gl_{k+M|N})$ is the inclusion which sends each $t_{ij}^{(r)}$
   in $Y(\gl_{M|N})$ to $t_{k+i,k+j}^{(r)}$ in $Y(\gl_{k+M|N})$, is an injective algebra homomorphism.
  \item [(4)]The map $\zeta_{M|N}:Y(\gl_{M|N})\rightarrow Y(\gl_{N|M})$ defined by
  \[
  \zeta_{M|N}=\rho_{M|N}\circ\omega_{M|N}
  \]
  is an algebra isomorphism.
\end{enumerate}
\end{proposition}
\begin{proof}
Follows by checking that these maps preserve the RTT relation (\ref{RTT}).
\end{proof}

\begin{remark} The composition $Y(\gl_N)\cong Y(\gl_{N|0})\xrightarrow{\zeta_{N|0}}Y(\gl_{0|N})$ is an algebra isomorphism.
\end{remark}

We call $\psi_k$ the $shift$ $map$ and $\zeta_{M|N}$ the $swap$ $map$. It is clear that $\psi_0$ is the identity map and $\zeta_{M|N}$ has order 2. Since they are important for us, we write down their images explicitly.
\begin{lemma}
Let $1\leq i,j \leq M+N$.

\begin{enumerate}
  \item [(1)]For any $k\in\mathbb{N}$, we have \begin{equation}\label{psit}
  \psi_k\big(t_{ij}(u)\big)=
  \left| \begin{array}{cccc} t_{11}(u) &\cdots &t_{1k}(u) &t_{1, k+j}(u)\\
         \vdots &\ddots &\vdots &\vdots \\
         t_{k1}(u) &\cdots &t_{kk}(u) &t_{k, k+j}(u)\\
         t_{k+i, 1}(u) &\cdots &t_{k+i,k}(u) &\boxed{t_{k+i, k+j}(u)}
         \end{array} \right|.
         \end{equation}
  \item [(2)]We have \begin{equation}\label{zetat}
        \zeta_{M|N}\big(t_{ij}(u)\big)=t_{M+N+1-i,M+N+1-j}^{\prime}(u).
        \end{equation}
\end{enumerate}
\end{lemma}
First note that the description of $\psi_k(t_{ij})$ in (\ref{psit}) is independent of M and N, hence our notation is unambiguous.
Also, (\ref{psit}) along with quasideterminants in Section~3 implies that
 \begin{eqnarray}\label{psid}
  D_{a;i,j}(u)&=&\psi_{\mu_1+\mu_2+\ldots +\mu_{a-1}}\big(D_{1;i,j}(u)\big),\\ \label{psie}
  E_{a;i,j}(u)&=&\psi_{\mu_1+\mu_2+\ldots +\mu_{a-1}}\big(E_{1;i,j}(u)\big),\\ \label{psif}
  F_{a;i,j}(u)&=&\psi_{\mu_1+\mu_2+\ldots +\mu_{a-1}}\big(F_{1;i,j}(u)\big).
 \end{eqnarray}

Secondly, observe that $\psi_k$ maps $t_{ij}^{\prime}(u)\in Y(\gl_{M|N})$ to $t_{k+i,k+j}^{\prime}(u)\in Y(\gl_{k+M|N})$. So $\psi_k\big(Y(\gl_{M|N})\big)$ is generated by the set
\{$t_{k+i,k+j}^{\prime(r)}\,|\,1\leq i,j\leq M+N, r\geq 0$\}, as a subalgebra of $Y(\gl_{k+M|N})$. If we pick any element $t_{ij}^{(r)}$ in the north-western $k\times k$ corner of $T(u)$ (viewed as an $(k+M+N)\times(k+M+N)$ matrix with entries in $Y(\gl_{k+M|N})[[u^{-1}]]$), the indices will never overlap with those of $\psi_k\big(Y(\gl_{M|N})\big)$, which are in the south-eastern $(M+N)\times(M+N)$ corner of the same $T(u)$. By equation (\ref{usefull}), they supercommute. Obviously, the elements in the north-western $k\times k$ corner in $Y(\gl_{k+M|N})$ generate a subalgebra isomorphic to $Y(\gl_k)$ by the defining relations (\ref{yangiandef}). We have proved the following lemma.
\begin{lemma}\label{corcommute}
The subalgebras $Y(\gl_k)$ and $\psi_k\big(Y(\gl_{M|N})\big)$ in $Y(\gl_{k+M|N})$ supercommute with each other.
\end{lemma}

Now we study the map $\zeta_{M|N}$. Associate to the composition $\mu$, we may define the elements \{$D_{a;i,j}^{(r)};D^{\prime(r)}_{a;i,j}$\}, \{$E_{a;i,j}^{(r)}$\}, \{$F_{a;i,j}^{(r)}$\} in $Y(\gl_{M|N})$ by Gauss decomposition. Consider
\[\mu^{r}:=(\mu_{m+n},\ldots,\mu_{m+1}\,|\,\mu_m,\ldots,\mu_{2},\mu_{1}),\]
the reverse of $\mu$, which is a composition of $(N|M)$.
With $\mu^r$, we may similarly define the elements \{$D_{a;i,j}^{(r)};D^{\prime(r)}_{a;i,j}$\}, \{$E_{a;i,j}^{(r)}$\}, \{$F_{a;i,j}^{(r)}$\} in $Y(\gl_{N|M})$, by abuse of notations. Their relations are given in the following proposition, which is a generalization of \cite[Proposition~1]{Go}.

\begin{proposition}\label{zetadef} For all admissible a, i, j, we have
\begin{eqnarray}
 \zeta_{M|N}\big(D_{a;i,j}(u)\big)&=&D^{\prime}_{m+n+1-a;\mu_{a}+1-i,\mu_{a}+1-j}(u),\label{zd}\\
 \zeta_{M|N}\big(E_{a;i,j}(u)\big)&=&-F_{m+n-a;\mu_{a}+1-i,\mu_{a+1}+1-j}(u),\label{sze}\\
 \zeta_{M|N}\big(F_{a;i,j}(u)\big)&=&-E_{m+n-a;\mu_{a+1}+1-i,\mu_{a}+1-j}(u)\label{szf}.
  \end{eqnarray}
  Note that the D's, E's and F's on the left hand side are in $Y(\gl_{M|N})[[u^{-1}]]$, while those on the right hand side are in $Y(\gl_{N|M})[[u^{-1}]]$.
\end{proposition}
\begin{proof}
The proof is essentially the same as \cite[Proposition~1]{Go}, except that we decompose the matrix $T(u)$ into block decompositions and the entry positions are flipped around by $\zeta$.
For a given composition $\mu$, multiply out the matrix products
\[ T(u)=F(u)D(u)E(u) \qquad \text{and} \qquad T(u)^{-1}=E(u)^{-1}D(u)^{\prime}F(u)^{-1}.\]
Then the following matrix identities hold.
\begin{eqnarray}
 \label{t11} T_{a,a}(u)&=&D_a(u)+\sum_{c<a}F_{a,c}(u)D_c(u)E_{c,a}(u),\\
 \label{t'11} T^{\prime}_{a,a}(u)&=&D^{\prime}_a(u)+\sum_{c>a}\widetilde{E}_{a,c}(u)D^{\prime}_c(u)\widetilde{F}_{c,a}(u),\\
 \label{t12} T_{a,b}(u)&=&D_a(u)E_{a,b}(u)+\sum_{c<a}F_{a,c}(u)D_c(u)E_{c,b}(u),\\
 \label{t21} T_{b,a}(u)&=&F_{b,a}(u)D_a(u)+\sum_{c<a}F_{b,c}(u)D_c(u)E_{c,a}(u),\\
 \label{t'12} T_{a,b}^{\prime}(u)&=&\widetilde E_{a,b}(u)D^{\prime}_b(u)+\sum_{c>b}\widetilde E_{a,c}(u)D^{\prime}_c(u)\widetilde F_{c,b}(u),\\
 \label{t'21} T_{b,a}^{\prime}(u)&=&D^{\prime}_b(u)\widetilde F_{b,a}(u)+\sum_{c>b}\widetilde E_{b,c}(u)D^{\prime}_c(u)\widetilde F_{c,a}(u),
\end{eqnarray}
for all $1\leq a\leq m+n$ in (\ref{t11}), (\ref{t'11}) and $1\leq a<b\leq m+n$ in (\ref{t12})$-$(\ref{t'21}). Here
$T_{a,b}^{\prime}(u)$ denotes the $\mu_a\times\mu_b$-matrices in the $(a,b)$-th block position of $T(u)^{-1}$, $T^{\prime}_{a,b;i,j}(u)$ denotes the $(i,j)$-th entry of $T_{a,b}^{\prime}(u)$, $T^{\prime (r)}_{a,b;i,j}$ denotes the coefficient of $u^{-r}$ in $T^{\prime}_{a,b;i,j}(u)$ and
\begin{eqnarray*}
\widetilde{E}_{a,b}(u)&:=&\sum_{a=i_0<i_1<\ldots<i_s=b}(-1)^sE_{a,i_1}(u)E_{i_1,i_2}(u)\cdots E_{i_{s-1},b}(u),\\
\widetilde{F}_{b,a}(u)&:=&\sum_{a=i_0<i_1<\ldots<i_s=b}(-1)^sF_{b,i_{s-1}}(u)F_{i_{s-1},i_{s-2}}(u)\cdots F_{i_{1},a}(u).
\end{eqnarray*}
In fact, (\ref{sze}) and (\ref{szf}) are the special cases when $b=a+1$ of the following more general relations.
\begin{eqnarray}
\zeta_{M|N}\big(E_{a,b;ij}(u)\big)&=&\widetilde{F}_{m+n+1-a,m+n+1-b;\mu_a+1-i,\mu_b+1-j}(u)\label{ze},\\
\zeta_{M|N}\big(F_{b,a;ij}(u)\big)&=&\widetilde{E}_{m+n+1-b,m+n+1-a;\mu_b+1-i,\mu_a+1-j}(u)\label{zf}.
\end{eqnarray}
One can easily derive (\ref{zd}), (\ref{ze}) and (\ref{zf}) simultaneously by induction on $a$.
\end{proof}

Now we describe the relations among the $D$'s. We first claim that
\[
[D_{a;i,j}(u),D_{b;h,k}(v)]=0,\qquad\text{ unless }\qquad a=b.
\]
Assume $a<b$. For $1\leq a\leq m$, there exists a suitable number $1\leq k\leq M$ such that $D_{a;i,j}(u)$ is contained in the north-western $k\times k$ corner of $Y(\gl_{M|N})[[u^{-1}]]$, i.e., 
\[
D_{a;i,j}(u)\in Y(\gl_k)[[u^{-1}]]\subset Y(\gl_{M|N})[[u^{-1}]]
\]
and
\[
D_{b;h,k}(v)\in \psi_k\big(Y(\gl_{M-k|N})\big)[[v^{-1}]]\subset Y(\gl_{M|N})[[v^{-1}]].
\]
Hence they supercommute by Lemma \ref{corcommute}. For $m+1\leq a\leq m+n$, we may apply the swap map $\zeta_{M|N}$ first then it is transformed to the above case in the super Yangian $Y(\gl_{N|M})$ and our claim follows.

We next compute the bracket explicitly when $a=b$. For $1\leq a\leq m$, by (\ref{psid}) and (\ref{Dtident}), we have
\begin{align*}
[D_{a;i,j}(u),D_{a;h,k}(v)]=&\psi_{\mu_1+\mu_2+\ldots +\mu_{a-1}}\big([D_{1;i,j}(u),D_{1;h,k}(v)]\big)\\
=&\psi_{\mu_1+\mu_2+\ldots +\mu_{a-1}}\big([t_{ij}(u),t_{hk}(v)]\big).
\end{align*}
For $m+1\leq a\leq m+n$, we set $\tilde{a}:=m+n+1-a$. 
Then we have $1\leq \tilde{a}\leq n$ and hence
\begin{align*}
D_{a;ij}(u)
&=\zeta_{N|M}\big( D_{\tilde{a};\mu_{\tilde{a}}+1-i,\mu_{\tilde{a}}+1-j}(u)  \big)\\
&=\zeta_{N|M}\circ\psi_{\mu_1+\mu_2+\ldots+\mu_{\tilde{a}-1}}
\big( D_{1;\mu_{\tilde{a}}+1-i,\mu_{\tilde{a}}+1-j}(u)  \big)\\
&=\zeta_{N|M}\circ\psi_{\mu_1+\mu_2+\ldots+\mu_{\tilde{a}-1}}
\big( t_{\mu_{\tilde{a}}+1-i,\mu_{\tilde{a}}+1-j}(u)\big).
\end{align*}
Therefore, for $m+1\leq a\leq m+n$, we have
\begin{align*}
[D_{a;i,j}(u),D_{a;h,k}(v)]=
&\,\zeta_{N|M}\circ\psi_{\mu_1+\mu_2+\ldots+\mu_{\tilde{a}-1}}\big([t_{\mu_{\tilde{a}}+1-i,\mu_{\tilde{a}}+1-j}(u),t_{\mu_{\tilde{a}}+1-h,\mu_{\tilde{a}}+1-k}(v)]\big).
\end{align*}
Referring to the definition (\ref{ydefs}), for any $1\leq a\leq m+n$, we have
\begin{equation*}
[D_{a;i,j}(u),D_{a;h,k}(v)]=\dfrac{1}{u-v}\big( D_{a;h,j}(u),D_{a;i,k}(v)-D_{a;h,j}(v),D_{a;i,k}(u)\big).
\end{equation*}
Collecting the coefficients of $u^{-r}v^{-s}$, we have proved the following proposition, which is parallel to the results in \cite[Section 4]{BK1}.

\begin{proposition}\label{MNdd0}
The relations among the elements
$\{D_{a;i,j}^{(r)},D_{a;i,j}^{\prime (r)}\}$ for all $r\geq 0$, ${1\leq i,j\leq \mu_a}$, $1\leq a\leq m+n$
are given by

\begin{eqnarray*}
    \label{D1}D_{a;i,j}^{(0)}=\delta_{ij}\:,\\
    \label{D2}\sum_{t=0}^{r}D_{a;i,p}^{(t)}D_{a;p,j}^{\prime (r-t)}=\delta_{r0}\delta_{ij}\:,
\end{eqnarray*}
\begin{equation*}\label{D3}
   [D_{a;i,j}^{(r)},D_{b;h,k}^{(s)}]=
    \delta_{ab}\sum_{t=0}^{min(r,s)-1}\big(D_{a;h,j}^{(t)}D_{a;i,k}^{(r+s-1-t)}
     -D_{a;h,j}^{(r+s-1-t)}D_{a;i,k}^{(t)}\big),
\end{equation*}
and these elements generate a subalgebra of $Y(\gl_{M|N})$.
\end{proposition}
 We call the subalgebra in Proposition \ref{MNdd0} the $standard$ $Levi$ $subalgebra$ of $Y(\gl_{M|N})$ associated to $\mu$ and denote it by $Y_{\mu}^0$. Note that in the special case when all $\mu_i=1$, the subalgebra $Y^0_{(1,\ldots,1)}$ is commutative.

\section{Special Cases: non-super case and m=n=1}
The following theorem of Brundan and Kleshchev describes the relations among the generators in the non-super case.
\begin{theorem}\label{bknonsuper}\cite[Theorem A]{BK1}
Let $\lambda=(\lambda_1,\lambda_2,\ldots,\lambda_m)$ be a composition of $M$. The following identities hold in
$Y(\gl_M)((u^{-1},v^{-1}))$ for all admissible $a,b,f,g,h,i,j,k$:.
\begin{align*}
(u-v)[D_{a;i,j}(u), E_{b;h,k}(v)]
        &=\delta_{a,b}\delta_{h,j}D_{a;i,p}(u)\big(E_{a;p,k}(v)-E_{a;p,k}(u)\big)\\
        &\qquad\qquad -\delta_{a,b+1}D_{a;i,k}(u)\big(E_{b;h,j}(v)-E_{b;h,j}(u)\big),\\[2mm]
(u-v)[D_{a;i,j}(u), F_{b;h,k}(v)]
        &=-\delta_{a,b}\delta_{k,i}\big(F_{b;h,p}(v)-F_{b;h,p}(u)\big)D_{a;p,j}(u)\\
        &\qquad\qquad +\delta_{a,b+1}\big(F_{b;i,k}(v)-F_{b;i,k}(u)\big)D_{a;h,j}(u),\\[2mm]
(u-v)[E_{a;i,j}(u), F_{b;h,k}(v)]
        &=\delta_{a,b}\big(D^{\prime}_{a;i,k}(u) D_{a+1;h,j}(u)-D_{a+1;h,j}(v) D^{\prime}_{a;i,k}(v)\big),\\[2mm]
(u-v)[E_{a;i,j}(u), E_{a;h,k}(v)]
        &=\big(E_{a;i,k}(u)-E_{a;i,k}(v)\big)\big(E_{a;h,j}(u)-E_{a;h,j}(v)\big),\\[2mm]
(u-v)[F_{a;i,j}(u), F_{a;h,k}(v)]
        &=\big(F_{a;i,k}(u)-F_{a;i,k}(v)\big)\big(F_{a;h,j}(u)-F_{a;h,j}(v)\big),\\[2mm]
(u-v)[E_{a;i,j}(u), E_{a+1;h,k}(v)]
        &=\delta_{h,j}\big(E_{a;i,q}(u)E_{a+1;q,k}(v)-E_{a;i,q}(v)E_{a+1;q,k}(v)\notag\\
        &\qquad\qquad\qquad\qquad\qquad +E_{a,a+2;i,k}(v)-E_{a,a+2;i,k}(u)\big),\\[2mm]
(u-v)[F_{a;i,j}(u), F_{a+1;h,k}(v)]
        &=\delta_{i,k}\big(-F_{a+1;h,q}(v)F_{a;q,j}(u)+F_{a+1;h,q}(v)F_{a;q,j}(v)\\
        &\qquad\qquad\qquad\qquad\qquad -F_{a+2,a;h,j}(v)+F_{a+2,a;h,j}(u)\big),\\[2mm]
(u-v)[E_{a;i,j}(u), E_{b;h,k}(v)]&= 0 \quad\text{ if\;\; $b>a+1$ \; or \; if \; $b=a+1$ \;and\;\; $h \neq j$},\\
(u-v)[F_{a;i,j}(u), F_{b;h,k}(v)]&= 0 \quad\text{ if\;\; $b>a+1$ \; or \; if \; $b=a+1$ \;and\;\; $i \neq k$},
\end{align*}
\begin{align*}
(u-v)\big[[E_{a;i,j}(u),E_{b;h,k}(v)],E_{b;f,g}(v)\big]&= 0\quad\:\:\text{ if $|a-b|\geq 1$},\\
(u-v)\big[E_{a;i,j}(u),[E_{a;h,k}(u),E_{b;f,g}(v)]\big]&= 0\quad\:\:\text{ if $|a-b|\geq 1$},
\end{align*}
\begin{align*}
\big[[E_{1;i,j}(u),E_{2;h,k}(v)],E_{2;f,g}(w)\big]+
\big[[E_{1;i,j}(u)&,E_{2;h,k}(w)],E_{2;f,g}(v)\big]=0\quad\text{ if $|a-b|\geq 1$},\\
\big[E_{1;i,j}(u),[E_{1;h,k}(v),E_{2;f,g}(w)]\big]+
\big[E_{1;i,j}(v),&[E_{1;h,k}(u),E_{2;f,g}(w)]\big]=0\quad\text{ if $|a-b|\geq 1$} ,
\end{align*}
where the index $p$ (resp. $q$) is summed over $1,\ldots,\lambda_a$ (resp. $1,\ldots,\lambda_{a+1}$).
\end{theorem}
\begin{proof}
See \cite[Section 6]{BK1}. Here, we present the theorem in the series form and we define the indices of $F$'s in a  slightly different manner.
\end{proof}

Back to the super case. Consider $m$=$n$=1; that is, $\mu=(\mu_1\,|\,\mu_2)=(M\,|\,N)$. Since we have only one block of $E$'s and $F$'s, we may omit the block positions without confusion. That is, we set
\[
\qquad E_{i,j}(u):=E_{1;i,j}(u)=E_{1,2;i,j}(u),\qquad
\text{for all}\quad 1\leq i\leq\mu_1=M,\quad 1\leq j\leq\mu_2=N,
\]
\[
\text{and}\;\; F_{i,j}(u):=F_{1;i,j}(u)=F_{2,1;i,j}(u),\qquad \text{for all}\quad 1\leq i\leq\mu_2=N,\quad 1\leq j\leq\mu_1=M.
\]
The relations among them are given in the following proposition, which is a generalization of \cite[Lemma 6.3]{BK1}.
\begin{proposition}\label{1stsuper}
The following identities hold in $Y(\gl_{M|N})((u^{-1},v^{-1}))$.
\begin{eqnarray}
\label{mnde}(u-v)[D_{a;i,j}(u),E_{h,k}(v)]&=&
                  \left\{\begin{array}{ll}
        \delta_{hj}D_{1;i,p}(u)\big(E_{p,k}(v)-E_{p,k}(u)\big), &{\text{ if }}\;a=1,\\[3mm]
        D_{2;i,k}(u)\big(E_{h,j}(v)-E_{h,j}(u)\big), &{\text{ if }}\;a=2,\\
       \end{array}\right.\\[2mm]
\label{mndf}(u-v)[D_{a;i,j}(u),F_{h,k}(v)]&=&
                  \left\{\begin{array}{ll}
        \delta_{ki}\big(F_{h,p}(u)-F_{h,p}(v)\big)D_{1;p,j}(u), &{\text{ if }}\;a=1,\\[3mm]
        \big(F_{i,k}(u)-F_{i,k}(v)\big)D_{2;h,j}(u), &{\text{ if }}\;a=2,\\
       \end{array}\right.\\[2mm]
\label{mnef}(u-v)[E_{i,j}(u),F_{h,k}(v)]&=&D^{\prime}_{1;i,k}(v)D_{2;h,j}(v)-D_{2;h,j}(u)D^{\prime}_{1;i,k}(u),\\
\label{mnee}(u-v)[E_{i,j}(u),E_{h,k}(v)]&=&\big(E_{i,k}(u)-E_{i,k}(v)\big)\big(E_{h,j}(v)-E_{h,j}(u)\big),\\
\label{mnff}(u-v)[F_{i,j}(u),F_{h,k}(v)]&=&\big(F_{i,k}(u)-F_{i,k}(v)\big)\big(F_{h,j}(v)-F_{h,j}(u)\big),
\end{eqnarray}
for all admissible $i,j,h,k$ and the index $p$ is summed over $1,\ldots, M$.
\end{proposition}

\begin{proof}
As in the proof of Proposition \ref{zetadef}, we compute the matrix product
\[T(u)=F(u)D(u)E(u)\qquad \text{and}\qquad T^{-1}(u)=E^{-1}(u)D^{\prime}(u)F^{-1}(u)\]
with respect to the composition $\mu=(M\,|\,N)$ and get the following identities.
\begin{eqnarray}
\label{541}t_{i,j}(u)&=&D_{1;i,j}(u), \qquad\qquad\qquad\quad\;\text{ for all }1\le i,j\le M,\\
\label{542}t_{i,M+j}(u)&=&D_{1;i,p}E_{p,j}(u),\qquad\qquad\quad\;\;\text{ for all } 1\le i\le M, 1\le j\le N,\\
\label{543}t_{M+i,j}(u)&=&F_{i,p}(u)D_{1;p,j}(u),\qquad\quad\quad\;\text{ for all } 1\le i\le N, 1\le j\le M,\\
\label{544}t_{M+i,M+j}(u)&=&F_{i,p}(u)D_{1;p,q}(u)E_{q,j}(u)+D_{2;i,j}(u),\quad\text{ for all } 1\le i,j\le N,\\
\label{545}t^{\prime}_{i,j}(u)&=&D^{\prime}_{1;i,j}(u)+E_{i,p'}(u)D^{\prime}_{2;p',q'}(u)F_{q',j}(u),\text{ for all } 1\le i,j\le M,\\
\label{546}t^{\prime}_{i,M+j}(u)&=&-E_{i,p'}(u)D^{\prime}_{2;p',j}(u),\qquad\quad\text{ for all } 1\le i\le M, 1\le j\le N,\\
\label{547}t^{\prime}_{M+i,j}(u)&=&-D^{\prime}_{2;i,p'}(u)F_{p',j}(u),\qquad\quad\,\text{ for all } 1\le i\le N, 1\le j\le M,\\
\label{548}t^{\prime}_{M+i,M+j}(u)&=&D^{\prime}_{2;i,j}(u),\qquad\qquad\qquad\quad\;\,\text{ for all } 1\le i,j\le N,
\end{eqnarray}
where the indices $p,q$ (resp. $p',q'$) are summed over $1,\ldots,M$ (resp. $1,\ldots,N$).

(\ref{mnde}) and (\ref{mndf}) can be proved using exactly the same method as in \cite[Lemma~6.3]{BK1} and hence we skip the detail.

To establish (\ref{mnef}), we need other identities. Computing the brackets in (\ref{mnde}) in the case $a=2$ and (\ref{mndf}) in the case $a=1$ and changing the indices, we have
\begin{equation}
(u-v)E_{\alpha ,j}(u)D^{\prime}_{2;h,\beta}(v)
  -\delta_{hj}\big(E_{\alpha ,q}(v)-E_{\alpha ,q}(u)\big)D^{\prime}_{2;q,\beta}(v)
=(u-v)D^{\prime}_{2;h,\beta}(v)E_{\alpha ,j}(u),
\end{equation}
\begin{equation}
-(u-v)F_{\beta ,k}(v)D_{1;i,\alpha}(u)
  +\delta_{ki}\big(F_{\beta ,p}(v)-F_{\beta ,p}(u)\big)D_{1;p,\alpha}(u)
=-(u-v)D_{1;i,\alpha}(u)F_{\beta ,k}(v),
\end{equation}
where $\alpha$, $p$ (resp. $\beta$, $q$) are summed over $1,\ldots,M$ (resp. $1,\ldots, N$).

By (\ref{usefull}), we have
\[
(u-v)[t_{i,M+j}(u),t^{\prime}_{M+h,k}(v)]=-\big(\delta_{hj}\sum_{l=1}^{M+N}t_{il}(u)t^{\prime}_{lk}(v)
-\delta_{ki}\sum_{s=1}^{M+N}t^{\prime}_{M+h,s}(v)t_{s,M+j}(u)\big).
\]
Substituting by (\ref{541})$-$(\ref{548}) and changing the indices, we may rewrite the above identity as the following
\begin{align}\notag
&D_{1;i,\alpha}(u)\big\lbrace(u-v)E_{\alpha ,j}(u)D^{\prime}_{2;h,\beta}(v)-\delta_{hj}\big(E_{\alpha ,q}(v)-E_{\alpha ,q}(u)\big)D^{\prime}_{2;q,\beta}(v)\big\rbrace F_{\beta ,k}(v)\\\notag
&\qquad-\delta_{hj}D_{1;i,\alpha}(u)D^{\prime}_{1;\alpha ,h}(v)\\\notag
&=D^{\prime}_{2;h,\beta}(v)\big\lbrace -(u-v)F_{\beta ,k}(v)D_{1;i,\alpha}(u)+\delta_{ki}\big(F_{\beta ,p}(v)-F_{\beta ,p}(u)\big)D_{1;p,\alpha}(u)\big\rbrace E_{\alpha ,j}(u)\\
&\qquad\qquad-\delta_{ki}D^{\prime}_{2;h,\beta}(v)D_{2;\beta ,j}(u),
\end{align}
where $\alpha$, $p$ (resp. $\beta$, $q$) are summed over $1,\ldots,M$(resp. $1,\ldots,N$). Substituting (5.14) and (5.15) into (5.16), we obtain
\begin{multline}
D_{1;i,\alpha}(u)\big\lbrace (u-v)D^{\prime}_{2;h,\beta}(v)E_{\alpha ,j}(u)F_{\beta ,k}(v)\big\rbrace
-\delta_{hj}D_{1;i,\alpha}(u)D^{\prime}_{1;\alpha ,k}(v)=\\
D^{\prime}_{2;h,\beta}(v)\big\lbrace -(u-v)D_{1;i,\alpha}(u)F_{\beta ,k}(v)E_{\alpha ,j}(u)\big\rbrace
-\delta_{ki}D^{\prime}_{2;h,\beta}(v)D_{2;\beta ,j}(u).
\end{multline}
Multiplying $D_2(v)D^{\prime}_{1}(u)$ from the left on both sides of (5.17), we obtain (\ref{mnef}).

For (\ref{mnee}), we start with $[t_{i,M+j}(u), t^{\prime}_{h,M+k}(v)]=0.$ Note that they are both odd elements.
Multiplying $(u-v)^2$ and computing the bracket after substitution by (\ref{542}) and (\ref{546}), we have
\begin{align}\notag
&(u-v)^2D_{1;i,p}(u)E_{p,j}(u)E_{h,q}(v)D^{\prime}_{2;q,k}(v)+\qquad\qquad\qquad\qquad\qquad\qquad\qquad\\
&\qquad\qquad\qquad(u-v)E_{h,q}(v)D_{1;i,p}(u)(u-v)D^{\prime}_{2;q,k}(v)E_{p,j}(u)=0.
\end{align}
Rewriting (\ref{mnde}) again, we have the following identities
\begin{align*}
(u-v)E_{h,q}(v)D_{1;i,p}(u)&=&(u-v)D_{1;i,p}(u)E_{h,q}(v)+ \delta_{hp}D_{1;i,p}(u)\big(E_{p,q}(u)-E_{p,q}(v)\big),\\
(u-v)D^{\prime}_{2;q,k}(v)E_{p,j}(u)&=&(u-v)E_{p,j}(u)D^{\prime}_{2;q,k}(v)+ \delta_{jq}\big(E_{p,q}(u)-E_{p,q}(v)\big)D^{\prime}_{2;q,k}(v).
\end{align*}
Substituting these two into the second term in (5.18) and multiplying $D_1(u)$ from the left, $D_2(v)$ from the right simultaneously, we obtain
\begin{multline}
(u-v)^2[E_{i,j}(u),E_{h,k}(v)]
=(u-v)E_{h,j}(v)\big(E_{i,k}(v)-E_{i,k}(u)\big)\\
+(u-v)\big(E_{i,k}(v)-E_{i,k}(u)\big)E_{h,j}(u)
+\big(E_{i,j}(u)-E_{i,j}(v)\big)\big(E_{h,k}(v)-E_{h,k}(u)\big).
\end{multline}
For a power series $P$ in $Y(\gl_{M|N})[[u^{-1}, v^{-1}]]$, we write $\big\lbrace P\big\rbrace_{d}$ for the homogeneous component of $P$ of total degree $d$ in the variables $u^{-1}$ and $v^{-1}$. (\ref{mnee}) follows from the following claim.

{\bf{Claim:}} For $d\ge 1$, we have
\[
(u-v)\big\lbrace [E_{i,j}(u), E_{h,k}(v)]\big\rbrace_{d+1}=\big\lbrace \big(E_{i,k}(u)-E_{i,k}(v)\big)\big(E_{h,j}(v)-E_{h,j}(u)\big)\big\rbrace_d.
\]

We prove the claim by induction on $d$.
For $d=1$, we take $\big\lbrace\:\big\rbrace_0$ on (5.19), and it implies
\[
(u-v)^2\big\lbrace [E_{i,j}(u),E_{h,k}(v)]\big\rbrace_2=0.
\]
Note that the right hand side of (5.19) is zero when $u=v$, hence we may divide both sides by $(u-v)$ and therefore $(u-v)\big\lbrace [E_{i,j}(u),E_{h,k}(v)]\big\rbrace_2=0$, as desired.
\\
Assume the claim is true for some $d>1$. By the hypothesis, we have
\begin{equation}
(u-v)\big\lbrace[E_{h,j}(u), E_{i,k}(v)]\big\rbrace_{d+1}=\big\lbrace\big(E_{h,k}(u)-E_{h,k}(v)\big)\big(E_{i,j}(v)-E_{i,j}(u)\big)\big\rbrace_d.
\end{equation}
\begin{equation*}\Longrightarrow \big\lbrace[E_{h,j}(u), E_{i,k}(v)]\big\rbrace_{d+1}=\Big\lbrace \frac{\big(E_{h,k}(u)-E_{h,k}(v)\big)\big(E_{i,j}(v)-E_{i,j}(u)\big)}{u-v}\Big\rbrace_d.\end{equation*}
Note that the right hand side is zero when $u=v$. Hence $\big\lbrace[E_{h,j}(v), E_{i,k}(v)]\big\rbrace_{d+1}=0$, which implies
\begin{equation}
E_{h,j}(v)E_{i,k}(v)=-E_{i,k}(v)E_{h,j}(v).
\end{equation}
Take $\big\lbrace\:\big\rbrace_d$ on (5.19):
\begin{align*}
(u-v)^2\big\lbrace [E_{i,j}(u), E_{h,k}(v)]\big\rbrace_{d+2}
&=(u-v)\big\lbrace E_{h,j}(v)\big(E_{i,k}(v)-E_{i,k}(u)\big)\big\rbrace_{d+1}\\
&\quad +(u-v)\big\lbrace \big(E_{i,k}(v)-E_{i,k}(u)\big)E_{h,j}(u)\big\rbrace_{d+1}\\
&\quad +\big\lbrace\big(E_{i,j}(u)-E_{i,j}(v)\big)\big(E_{h,k}(v)-E_{h,k}(u)\big)\big\rbrace_d.
\end{align*}
Substituting the last term by (5.20) and simplifying the result, we have
\begin{multline*}
(u-v)^2\big\lbrace [E_{i,j}(u), E_{h,k}(v)]\big\rbrace_{d+2}=\\
(u-v)\big\lbrace E_{h,j}(v)E_{i,k}(v)+E_{i,k}(u)E_{h,j}(v)
+\big(E_{i,k}(v)-E_{i,k}(u)\big)E_{h,j}(u)\big\rbrace_{d+1}.
\end{multline*}
Substituting by (5.21) into the above identity, we have
\begin{align*}
&(u-v)^2\big\lbrace [E_{i,j}(u), E_{h,k}(v)]\big\rbrace_{d+2}\\
&=(u-v)\big\lbrace \big(E_{i,k}(u)-E_{i,k}(v)\big)E_{h,j}(v)-\big(E_{i,k}(u)-E_{i,k}(v)\big)E_{h,j}(u)\big\rbrace_{d+1}\\
&=(u-v)\big\lbrace \big(E_{i,k}(u)-E_{i,k}(v)\big)\big(E_{h,j}(v)-E_{h,j}(u)\big)\big\rbrace_{d+1}.
\end{align*}
Dividing both sides by $u-v$ establishes the claim.

(\ref{mnff}) follows from applying the map $\zeta_{N|M}$ to (\ref{mnee}) in $Y(\gl_{N|M})[[u^{-1},v^{-1}]]$ with suitable indices.
\end{proof}

\section{Special Case: m=2, n=1}
Recall that $m$ is the number of parts of the composition of $M$ and $n$ is the number of parts of the composition of $N$. In the case when $m=2$, $n=1$, $\mu=(\mu_1,\mu_2\,|\,\mu_3)$, where $\mu_1+\mu_2=M$ and $\mu_3=N$. The relations among $E_{a;i,j}(u)$ and $F_{b;h,k}(u)$ in different blocks are obtained by the following lemma, which is a generalization of \cite[Lemma~6.4]{BK1} and \cite[Lemma~3]{Go}.

Before stating and proving the lemma, we first set a notation for the remaining part of this article. We denote the super Yangian by the notation \[Y_{\mu}:=Y(\gl_{M|N})\] to emphasize how we decompose the matrix $T(u)$ into block matrices according to the composition $\mu$ of $(M|N)$ and how those $D$'s, $E$'s and $F$'s are defined. Moreover, by abuse of notations, we will consider the $D$'s, $E$'s and $F$'s in different super Yangians at the same time. It should be clear from the context which super Yangian we are dealing with. 
\begin{lemma}\label{mmn}
The following identities hold in $Y_{(\mu_1,\mu_2|\mu_3)}((u^{-1},v^{-1}))$ for all admissible $g,h,i,j,k$.
\begin{itemize}
\item[(a)] $[E_{1;i,j}(u), F_{2;h,k}(v)] = 0$,
\item[(b)] \hspace*{-2mm}$[E_{1;i,j}(u), E_{2;h,k}(v)] =
\dfrac{\delta_{hj}}{u-v}\lbrace\big(E_{1;i,q}(u)-E_{1;i,q}(v)\big)E_{2;q,k}(v)
+ E_{1,3;i,k}(v) - E_{1,3;i,k}(u)\rbrace$,
\item[(c)]
$[E_{1,3;i,j}(u), E_{2;h,k}(v)] = E_{2;h,j}(v) [E_{1;i,g}(u), E_{2;g,k}(v)]$,
\item[(d)]
$[E_{1;i,j}(u), E_{1,3;h,k}(v) - E_{1;h,q}(v) E_{2;q,k}(v)]
= -[E_{1;i,g}(u), E_{2;g,k}(v)] E_{1;h,j}(u)$.
\end{itemize}
Here, $q$ is summed over $1,\dots,\mu_2$ and $g$ could be any number in $\lbrace 1,2,\ldots,\mu_{2}\rbrace$.
\end{lemma}

\begin{proof}
($a$) By (\ref{usefull}), we have $[t_{i,\mu_1+j}(u),t^{\prime}_{\mu_1+\mu_2+h,\mu_1+k}(v)]=0$. Substituting by $(\ref{t11})-(\ref{t'21})$ with respect to the composition $\mu$ and according to the indices, we have
\[
[D_{1;i,p}(u)E_{1;p,j}(u),-D^{\prime}_{3;h,q}(v)F_{2;q,k}(v)]=0.
\]
Computing the bracket, we obtain
\begin{equation}
D_{1;i,p}(u)E_{1;p,j}(u)D^{\prime}_{3;h,q}(v)F_{2;q,k}(v)-D^{\prime}_{3;h,q}(v)F_{2;q,k}(v)D_{1;i,p}(u)E_{1;p,j}(u)=0,
\end{equation}
where $p$ and $q$ are summed over $1,\ldots,\mu_1$ and $1,\ldots,\mu_3$, respectively. Similarly, by (\ref{usefull}), we have
\[
[t_{ij}(u),t^{\prime}_{\mu_1+\mu_2+h,\mu_1+k}(v)]=[t_{i,\mu_1+j}(u),t^{\prime}_{\mu_1+\mu_2+h,\mu_1+\mu_2+k}(v)]=0,
\]
which implies that
\[
[D_{1;i,j}(u),F_{2;h,k}(v)]=[E_{1;i,j}(u),D^{\prime}_{3;h,k}(v)]=0.
\]
Substituting these into (6.1) and noting that $[D_{1;i,j}(u),D^{\prime}_{3;h,k}(v)]=0$, we have
\[
D_{1;i,p}(u)D^{\prime}_{3;h,q}(v)E_{1;p,j}(u)F_{2;q,k}(v)
-D_{1;i,p}(u)D^{\prime}_{3;h,q}(v)F_{2;q,k}(v)E_{1;p,j}(u)=0.
\]
Multiplying $D_{3}(v)D^{\prime}_{1}(u)$ from the left, we obtain $(a)$.

($b$) By (\ref{usefull}), we have
\[
(u-v)[t_{i,\mu_1+j}(u),t^{\prime}_{\mu_1+h,\mu_1+\mu_2+k}(v)]=\delta_{jh}\sum_{s=1}^{M+N}t_{is}(u)t_{s,\mu_1+\mu_2+k}(v).
\]
Substituting by (\ref{t11})$-$(\ref{t'21}) according to the indices in the above identity, we have
\begin{multline}
(u-v)[D_{1;i,p}(u)E_{1;p,j}(u),-E_{2;h,q}(v)D^{\prime}_{3;q,k}(v)]=\\
\delta_{jh}D_{1;i,p}(u)\big\lbrace \big(E_{1;p,r}(v)E_{2;r,q}(v)-E_{1,3;p,q}(v)\big)-E_{1;p,r}(u)E_{2;r,q}(v)+E_{1,3;p,q}(u)\big\rbrace D^{\prime}_{3;q,k}(v),
\end{multline}
where the indices $p,q,r$ are summed over $\mu_1,\mu_3,\mu_2$, respectively.
Using the facts that
\begin{eqnarray*}
\big[E_{1;i,j}(v),D^{\prime}_{3;h,k}(u)\big]=0,& \qquad\big(\text{explained in the proof of }(a)\big)\\
\big[E_{2;i,j}(v),D_{1;h,k}(u)\big]=0,& \qquad \big(\text{obtained from}\;\;[t_{ij}(u),t^{\prime}_{\mu_1+h,\mu_1+\mu_2+k}(v)]=0\big)
\end{eqnarray*}
we may cancel $D_{1}(u)$ from the left and $D^{\prime}_3(v)$ from the right on both sides of (6.2). Dividing both sides by $u-v$, we have proved $(b)$.

($c$) By (5.1) in $Y_{(\mu_2|\mu_3)}[[u^{-1},v^{-1}]]$, we have
\begin{equation*}
(u-v)[E_{1;h,k}(u),D^{\prime}_{2;i,j}(v)]=\delta_{ki}\big(E_{1;h,p}(v)-E_{1;h,p}(u)\big)D^{\prime}_{2;p,j}(v).
\end{equation*}
Applying the map $\psi_{\mu_1}$ to this identity and using (4.3)$-$(4.5), we have the following identity in $Y_{(\mu_1,\mu_2|\mu_3)}[[u^{-1},v^{-1}]]$
\begin{equation*}
(u-v)[E_{2;h,k}(u),D^{\prime}_{3;i,j}(v)]=\delta_{ki}\big(E_{2;h,p}(v)-E_{2;h,p}(u)\big)D^{\prime}_{3;p,j}(v).
\end{equation*}
Taking the coefficient of $u^0$, we obtain
\begin{equation}
[E_{2;h,k}^{(1)},D^{\prime}_{3;i,j}(v)]=\delta_{ki}E_{2;h,p}(v)D^{\prime}_{3;p,j}(v).
\end{equation}
Also by (\ref{ter}), we have 
\begin{equation}
E_{1,3;i,j}(u)=[E_{1;i,g}(u),E^{(1)}_{2;g,j}],\; \text{for any} \; 1\leq g\leq\mu_2.
\end{equation}
By (6.3), (6.4) and the fact that $[E_{1;i,g}(u),D^{\prime}_{3;h,k}(v)]=0$, we have
\begin{align}\notag
[E_{1,3;i,j}(u),D^{\prime}_{3;h,k}(v)]&=\big[[E_{1;i,g}(u),E^{(1)}_{2;g,j}\big],D^{\prime}_{3;h,k}(v)]\\\notag
&=\big[E_{1;i,g}(u),[E^{(1)}_{2;g,j},D^{\prime}_{3;h,k}(v)]\big]\\\notag
&=\big[E_{1;i,g}(u),\delta_{hj}E_{2;g,p}(v)D^{\prime}_{3;p,k}(v)\big]\\
&=\delta_{hj}\big[E_{1;i,g}(u),E_{2;g,p}(v)\big]D^{\prime}_{3;p,k}(v). 
\end{align}
By (\ref{usefull}) and (\ref{t11})$-$(\ref{t'21}), we have
\[
[t_{i,\mu_1+\mu_2+j}(u),t^{\prime}_{\mu_1+h,\mu_1+\mu_2+k}(v)]=[D_{1;i,p}(u)E_{1,3;p,j}(u),-E_{2;h,q}(v)D^{\prime}_{3;q,k}(v)]=0,
\]
where $p$ and $q$ are summed over $1,2,\ldots,\mu_1$ and $1,2,\ldots,\mu_3$, respectively. Multiplying $D_1^{\prime}(u)$ from the left, we have $[E_{1,3;i,j}(u),E_{2;h,q}(v)D^{\prime}_{3;q,k}(v)]=0,$ which may be written as
\[
[E_{1,3;i,j}(u),E_{2;h,q}(v)]D^{\prime}_{3;q,k}(v)-E_{2;h,q}(v)[E_{1,3;i,j}(u),D^{\prime}_{3;q,k}(v)]=0.
\]
Substituting the last bracket by (6.5), we have
\[
[E_{1,3;i,j}(u),E_{2;h,q}(v)]D^{\prime}_{3;q,k}(v)-\delta_{qj}E_{2;h,q}(v)[E_{1;i,g}(u),E_{2;g,p}(v)]D^{\prime}_{3;p,k}(v)=0.
\]
\[
\Longrightarrow
[E_{1,3;i,j}(u),E_{2;h,q}(v)]D^{\prime}_{3;q,k}(v)=E_{2;h,j}(v)[E_{1;i,g}(u),E_{2;g,p}(v)]D^{\prime}_{3;p,k}(v).
\]
Multiplying $D_{3}(v)$ from the right to both sides of the above equality, we obtain ($c$).

($d$) Taking the coefficient of $u^0$ in ($b$), we have
\begin{equation}\notag
[E_{1;i,j}^{(1)},E_{2;h,k}(v)]=\delta_{hj}\big(E_{1,3;i,k}(v)-E_{1;i,q}(v)E_{2;q,k}(v)\big).
\end{equation}
Taking the coefficient of $v^0$ in (5.1) in the case $a=1$, we have
\begin{equation}\notag
[D_{1;i,j}(u),E_{1;h,k}^{(1)}]=\delta_{hj}D_{1;i,p}(u)E_{1;p,k}(u).
\end{equation}
By the above two equalities and the fact that $[D_{1;i,j}(u),E_{2;g,k}(v)]=0$, we have
\begin{align}
[D_{1;i,j}(u),E_{1,3;h,k}(v)-E_{1;h,q}(v)E_{2;q,k}(v)]&=[D_{1;i,j}(u),\big[E_{1;h,g}^{(1)},E_{2;g,k}(v)]\big]\notag\\
&=\big[[D_{1;i,j}(u),E_{1;h,g}^{(1)}],E_{2;g,k}(v)\big]\notag\\
&=[\delta_{hj}D_{1;i,p}(u)E_{1;p,g}(u),E_{2;g,k}(v)]\notag\\
&=\delta_{hj}D_{1;i,p}(u)[E_{1;p,g}(u),E_{2;g,k}(v)].
\end{align}
Taking the sum of all $j$ in (6.6), we have
\begin{multline}\notag
\delta_{hr}D_{1;i,p}(u)[E_{1;p,g}(u)E_{2;g,k}(v)]=
D_{1;i,r}(u)\big(E_{1,3;h,k}(v)-E_{1;h,s}(v)E_{2;s,k}(v)\big)\\
-\big(E_{1,3;h,k}(v)-E_{1;h,s}(v)E_{2;s,k}(v)\big)D_{1;i,r}(u),
\end{multline}
where $p,r,s$ are summed over $\mu_1,\mu_1,\mu_2$, respectively. Changing the indices, we may rewrite the above equality as
\begin{multline}
\big(E_{1;h,r}(v)E_{2;r,k}(v)-E_{1,3;h,k}(v)\big)D_{1;i,p}(u)=\\
\delta_{hp}D_{1;i,p'}(u)[E_{1;p',g}(u),E_{2;g,k}(v)]
+D_{1;i,p}(u)\big(E_{1;h,r}(v)E_{2;r,k}(v)-E_{1,3;h,k}(v)\big),
\end{multline}
where $r,p,p'$ are summed over $\mu_2, \mu_1 ,\mu_1$, respectively.

On the other hand, by (\ref{usefull}) and (\ref{t11})$-$(\ref{t'21}), we have
\begin{multline}
[t_{i,\mu_1+j}(u),t^{\prime}_{h,\mu_1+\mu_2+k}(v)]=\\
[D_{1;i,p}(u)E_{1;p,j}(u),\big(E_{1;h,r}(v)E_{2;r,q}(v)-E_{1,3;h,q}(v)\big)D^{\prime}_{3;q,k}(v)]=0,
\end{multline}
where $p$ and $q$ are summed over $\mu_1$ and $\mu_3$, respectively. Multiplying $D_{3}(v)$ from the right and computing the bracket, (6.8) becomes
\begin{multline}
D_{1;i,p}(u)E_{1;p,j}(u)\big(E_{1;h,r}(v)E_{2;r,k}(v)-E_{1,3;h,k}(v)\big)\\
-\big(E_{1;h,r}(v)E_{2;r,k}(v)-E_{1,3;h,k}(v)\big)D_{1;i,p}(u)E_{1;p,j}(u)=0,
\end{multline}
where $p$ and $r$ are summed over $\mu_1$ and $\mu_2$, respectively. Substituting (6.7) into the second term of (6.9), we have
\begin{multline*}
D_{1;i,p}(u)E_{1;p,j}(u)\big(E_{1;h,q}(v)E_{2;q,k}(v)-E_{1,3;h,k}(v)\big)\\
-\delta_{h,p_1}D_{1;i,p_2}(u)\big[E_{1;p_2,g}(u),E_{2;g,k}(v)\big]E_{1;p_1,j}(u)\\
-D_{1;i,p_3}(u)\big(E_{1;h,q_1}(v)E_{2;q_1,k}(v)-E_{1,3;h,k}(v)\big)E_{1;p_3,j}(u)=0.
\end{multline*}
Multiplying $D^{\prime}_{1}(u)$ from the left, we obtain
\begin{multline*}
E_{1;i,j}(u)\big(E_{1;h,q}(v)E_{2;q,k}(v)-E_{1,3;h,k}(v)\big)\\
-\big[E_{1;i,g}(u),E_{2;g,k}(v)\big]E_{1;h,j}(u)
-\big(E_{1;h,q_1}(v)E_{2;q_1,k}(v)-E_{1,3;h,k}(v)\big)E_{1;i,j}(u)=0.
\end{multline*}
Simplifying the above, we obtain ($d$).
\end{proof}


We have the F-counterpart of Lemma \ref{mmn}.
\begin{lemma}
The following identities hold in $Y_{(\mu_1,\mu_2|\mu_3)}((u^{-1},v^{-1}))$ for all admissible $g,h,i,j,k$.
\begin{itemize}
\item[(a)]
$[F_{1;i,j}(u),E_{2;h,k}(v)]=0,$
\item[(b)]
\hspace*{-2mm}$[F_{1;i,j}(u),F_{2;h,k}(v)]=
\dfrac{\delta_{ik}}{u-v}\big\lbrace F_{2;h,q}(v)\big(F_{1;q,j}(v)-F_{1;q,j}(u)\big)-F_{3,1;h,j}(v)+F_{3,1;h,j}(u)\big\rbrace,$
\item[(c)]
$[F_{3,1;i,j}(u),F_{2;h,k}(v)]=[F_{2;h,g}(v),F_{1;g,j}(u)]F_{2;i,k}(v),$
\item[(d)]
$[F_{1;i,j}(u) , F_{2;h,q}(v)F_{1;q,k}(v)-F_{3,1;h,k}(v)]=F_{1;i,k}(u)[F_{1;g,j}(u),F_{2;h,g}(v)].$
\end{itemize}
Here, $q$ is summed over $1,\dots,\mu_2$ and $g$ could be any number in $\lbrace 1,2,\ldots,\mu_{2}\rbrace$.
\end{lemma}
\begin{proof}
They can be proved by similar methods as in the proof of Lemma 6.1 and we skip the details.
\end{proof}
The following lemma is a generalization of \cite[Lemma 6.5, Lemma 6.6]{BK1} and of part of \cite[Lemma~3]{Go}.
\begin{lemma}
The following identities hold in $Y_{(\mu_1,\mu_2|\mu_3)}[[u^{-1},v^{-1},w^{-1}]]$ for all admissible $f,g,h,i,j,k$.
\begin{itemize}
\item[(a)]$\big[[E_{1;i,j}(u),E_{2;h,k}(v)],E_{2;f,g}(v)\big]=0,$
\item[(b)]$\big[E_{1;i,j}(u),[E_{1;h,k}(u),E_{2;f,g}(v)]\big]=0,$
\item[(c)]$\big[[E_{1;i,j}(u),E_{2;h,k}(v)],E_{2;f,g}(w)\big]+
\big[[E_{1;i,j}(u),E_{2;h,k}(w)],E_{2;f,g}(v)\big]=0,$
\item[(d)]$\big[E_{1;i,j}(u),[E_{1;h,k}(v),E_{2;f,g}(w)]\big]+
\big[E_{1;i,j}(v),[E_{1;h,k}(u),E_{2;f,g}(w)]\big]=0,$
\item[(e)]
$\big[[F_{1;i,j}(u),F_{2;h,k}(v)],F_{2;f,g}(v)\big]=0,$
\item[(f)]
$\big[F_{1;i,j}(u),[F_{1;h,k}(u),F_{2;f,g}(v)]\big]=0,$
\item[(g)]
$\big[[F_{1;i,j}(u),F_{2;h,k}(v)],F_{2;f,g}(w)\big]+\big[[F_{1;i,j}(u),F_{2;h,k}(w)],F_{2;f,g}(v)\big]=0,$
\item[(h)]
$\big[F_{1;i,j}(u),[F_{1;h,k}(v),F_{2;f,g}(w)]\big]+\big[F_{1;i,j}(v),[F_{1;h,k}(u),F_{2;f,g}(w)]\big]=0.$
\end{itemize}
\end{lemma}

\begin{proof}
We prove ($a$) and ($c$) in detail here, while the others can be proved in a similar fashion.

($a$) We first claim that
\[
[E_{a;i,j}(v),E_{a;h,k}(v)]=0 \quad \text{for}\; a=1,2 \quad\text{in}\quad Y_{(\mu_1,\mu_2|\mu_3)}[[u^{-1},v^{-1}]].
\]
The case $a=1$ follows from Theorem \ref{bknonsuper} and $a=2$ follows from applying the map $\psi_{\mu_1}$ to (\ref{mnee}).

By the super-Jacobi identity, together with the claim and Lemma \ref{mmn}($b$), it suffices to prove the case when $j=h=f$. In this case, we compute the bracket by Lemma \ref{mmn} as follows. 
\begin{align*}
(u-v)\big[[E&_{1;i,j}(u),E_{2;j,k}(v)],E_{2;j,g}(v)\big]\\
&=-(u-v)\big[E_{2;j,k}(v),[E_{1;i,j}(u),E_{2;j,g}(v)]\big]\\
&=[E_{1;i,q}(u) E_{2;q,g}(v)-E_{1;i,q}(v) E_{2;q,g}(v)
+ E_{1,3;i,g}(v) - E_{1,3;i,g}(u),E_{2;j,k}(v)]\\
&=[E_{1;i,q}(u)E_{2;q,g}(v),E_{2;j,k}(v)]+[E_{1,3;i,g}(v),E_{2;j,k}(v)]\\
&\qquad -[E_{1;i,q}(v)E_{2;q,g}(v),E_{2;j,k}(v)]-[E_{1,3;i,g}(u),E_{2;j,k}(v)]\\
&=-[E_{1;i,q}(u),E_{2;j,k}(v)]E_{2;q,g}(v)-E_{2;j,g}(v)[E_{1;i,j}(u),E_{2;j,k}(v)]\\
&\qquad +[E_{1;i,q}(v),E_{2;j,k}(v)]E_{2;q,g}(v)+E_{2;j,g}(v)[E_{1;i,j}(u),E_{2;j,k}(v)]\\
&=-\big[[E_{1;i,j}(u),E_{2;j,k}(v)],E_{2;j,g}(v)\big]+\big[[E_{1;i,j}(v),E_{2;j,k}(v)],E_{2;j,g}(v)\big].
\end{align*}
Thus we have
\begin{equation}
(u-v-1)\big[[E_{1;i,j}(u),E_{2;j,k}(v)],E_{2;j,g}(v)\big]=-\big[[E_{1;i,j}(v),E_{2;j,k}(v)],E_{2;j,g}(v)\big].
\end{equation}
Note that the right hand side of (6.10) is independent of the choice of $u$. Set $u=v+1$, then the right hand side  of (6.10) is zero. Using (6.10) again, we obtain ($a$).

($c$) It is enough to show that
\begin{equation}
(u-w)(v-w)(u-v)\big[[E_{1;i,j}(u),E_{2;h,k}(v)],E_{2;f,g}(w)\big]
\end{equation}
is symmetric in $v$ and $w$. We may further assume $j=h$, as in the proof of ($a$).
By Lemma \ref{mmn}$(b)$, we have
\begin{align*}
&(u-v)\big[[E_{1;i,j}(u),E_{2;j,k}(v)],E_{2;f,g}(w)\big]\\
&=\big[E_{1;i,q}(u)E_{2;q,k}(v)-E_{1;i,q}(v)E_{2;q,k}(v)+E_{1,3;i,k}(v)-E_{1,3;i,k}(u),E_{2;f,g}(w)\big].
\end{align*}
Multiplying both sides with $(u-w)(v-w)$, computing the brackets by Lemma~\ref{mmn}, we have
{\allowdisplaybreaks
\begin{align}\notag
(u-w&)(v-w)(u-v)\big[[E_{1;i,j}(u),E_{2;h,k}(v)],E_{2;f,g}(w)\big]\\ \notag
=&(u-w)(v-w)\big\lbrace E_{1;i,q}(u)E_{2;q,k}(v)E_{2;f,g}(w)+E_{2;f,g}(w)E_{1;i,q}(u)E_{2;q,k}(v)\\\notag
\,&-E_{1;i,q}(v)E_{2;q,k(v)}E_{2;f,g}(w)-E_{2;f,g}(w)E_{1;i,q}(v)E_{2;q,k}(v)\\\notag
\,&+E_{2;f,k}(w)[E_{1;i,x}(v),E_{2;x,g}(w)]-E_{2;f,k}(w)[E_{1;i,x}(u),E_{2;x,g}(w)]\big\rbrace\\\notag
=&(u-w)(v-w)\big\lbrace E_{1;i,q}(u)[E_{2;q,k}(v),E_{2;f,g}(w)]-E_{1;i,q}(u)E_{2;f,g}(w)E_{2;q,k}(v)\\\notag
\,&+[E_{2;f,g}(w),E_{1;i,q}(u)]E_{2;q,k}(v)+E_{1;i,q}(u)E_{2;f,g}(w)E_{2;q,k}(v)\\\notag
\,&-E_{1;i,q}(v)[E_{2;q,k}(v),E_{2;f,g}(w)]+E_{1;i,q}(v)E_{2;f,g}(w)E_{2;q,k}(v)\\\notag
\,&-[E_{2;f,g}(w),E_{1;i,q}(v)]E_{2;q,k}(v)-E_{1;i,q}(v)E_{2;f,g}(w)E_{2;q,k}(v)\\\notag
\,&-[E_{1;i,x}(v),E_{2;x,g}(w)]E_{2;f,k}(w)+[E_{1;i,x}(u),E_{2;x,g}(w)]E_{2;f,k}(w)\big\rbrace\\\notag
=&(u-w)(v-w)E_{1;i,q}(u)\big[E_{2;q,k}(v),E_{2;f,g}(w)\big]\\\notag
\,&+(u-w)(v-w)\big[E_{2;f,g}(w),E_{1;i,q}(u)\big]E_{2;q,k}(v)\\\notag
\,&-(u-w)(v-w)E_{1;i,q}(v)\big[E_{2;q,k}(v),E_{2;f,g}(w)\big]\\\notag
\,&-(u-w)(v-w)\big[E_{2;f,g}(w),E_{1;i,q}(v)\big]E_{2;q,k}(v)\\\notag
\,&-(u-w)(v-w)\big[E_{1;i,x}(v),E_{2;x,g}(w)\big]E_{2;f,k}(w)\\
\,&+(u-w)(v-w)\big[E_{1;i,x}(u),E_{2;x,g}(w)\big]E_{2;f,k}(w).
\end{align}}
Now we use (\ref{mnee}) and Lemma \ref{mmn} to compute these brackets, then (6.12) equals
\begin{align*}
&(u-w)E_{1;i,q}(u)\big(E_{2;q,g}(v)-E_{2;q,g}(w)\big)\big(E_{2;f,k}(w)-E_{2;f,k}(v)\big)\\
\,&-(v-w)\delta_{q,j}\big\lbrace \big(E_{1;i,q_0}(u)-E_{1;i,q_0}(w)\big)E_{2;q_0,g}(w)+E_{1,3;i,g}(w)-E_{1,3;i,g}(u)\big\rbrace E_{2;q,k}(v)\\
\,&-(u-w)E_{1;i,q}(v)\big(E_{2;q,g}(v)-E_{2;q,g}(w)\big)\big(E_{2;f,k}(w)-E_{2;f,k}(v)\big)\\
\,&+(u-w)\delta_{q,f}\big\lbrace \big(E_{1;i,q_0}(v)-E_{1;i,q_0}(w)\big)E_{2;q_0,g}(w)+E_{1,3;i,g}(w)-E_{1,3;i,g}(v)\big\rbrace E_{2;q,k}(v)\\
\,&-(u-w)\big\lbrace E_{1;i,q}(v)E_{2;q,g}(w)-E_{1;i,q}(w)E_{2;q,g}(w)+E_{1,3;i,g}(w)-E_{1,3;i,g}(v)\big\rbrace E_{2;f,k}(w)\\
\,&+(v-w)\big\lbrace E_{1;i,q}(u)E_{2;q,g}(w)-E_{1;i,q}(w)E_{2;q,g}(w)+E_{1,3;i,g}(w)-E_{1,3;i,g}(u)\big\rbrace E_{2;f,k}(w),
\end{align*}
where the indices $q$ and $q_0$ are summed over $1,2,\ldots\mu_2$.

Opening the parentheses of the above equality, we obtain that the resulting expression is indeed symmetric in $v$ and $w$. Therefore, (6.11) is symmetric in $v$ and $w$ and hence ($c$) is proved.
\end{proof}

\section{The general Case}

Recall that our goal is to obtain the relations among the generators $\lbrace D_{a;i,j}^{(r)}, D_{a;i,j}^{\prime'(r)}\rbrace$, $\lbrace E_{a;i,j}^{(r)}\rbrace$, and $\lbrace F_{a;i,j}^{(r)}\rbrace$ associated to a composition $\mu$ of $(M|N)$. To that end, we divide them into 3 disjoint parts as following:
\begin{itemize}
\item[A\,:]$\big\lbrace D_{a;i,j}^{(r)}, D_{a;i,j}^{\prime(r)}\big\rbrace_{1\leq a\leq m}
\cup \big\lbrace E_{a;i,j}^{(r)}\big\rbrace_{1\leq a< m}\cup
\big\lbrace F_{a;i,j}^{(r)}\big\rbrace_{1\leq a< m}$,

\item[B\,:]$\big\lbrace D_{a;i,j}^{(r)}, D_{a;i,j}^{\prime(r)}\big\rbrace_{m+1\leq a\leq m+n}
\cup\big\lbrace E_{a;i,j}^{(r)}\big\rbrace_{m+1\leq a< m+n}\cup
\big\lbrace F_{a;i,j}^{(r)}\big\rbrace_{m+1\leq a< m+n}$,

\item[C\,:]$\big\lbrace E_{m;i,j}^{(r)}\big\rbrace\cup
\big\lbrace F_{m;i,j}^{(r)}\big\rbrace$,
\end{itemize}
for all admissible indices $i, j, r$.

If we choose two elements from Part A, then their bracket is obtained by Theorem~\ref{bknonsuper}. If we choose two elements from Part B, then they are the images of some elements from the Part A in $Y(\gl_{N|M})$ under the swap map $\zeta_{N|M}$, and the bracket is obtained by Theorem~\ref{bknonsuper} as well.

Now suppose one of them is from Part A and the other is from Part B. Note that every element in Part A is in the north-western $M\times M$ corner of $T(u)$ and hence is in the subalgebra $Y(\gl_M)$ of $Y(\gl_{M|N})$ (see Section 4). On the other hand, every element in Part B is in the south-eastern $N\times N$ corner of $T(u)$ and hence is in the subalgebra $\psi_M\big(Y(\gl_{0|N})\big)$ of $Y(\gl_{M|N})$. Thus, their bracket is zero by Lemma \ref{corcommute}.

 Therefore, we only have to focus on the cross section where the odd blocks and even blocks are ``close", and this is done in Proposition~\ref{1stsuper}, Lemma \ref{mmn} and Lemma 6.2. Moreover, there are some non-trivial ternary brackets relations in the non-super case, and the corresponding ternary relations in the super case are found in Lemma 6.3.

The following proposition summarizes the results we have obtained up to now.
{\allowdisplaybreaks
\begin{proposition}\label{pgpropo}
For all admissible $a,b,f,g,h,i,j,k$, we have the following equalities in the super Yangian $Y_{\mu}((u^{-1},v^{-1},w^{-1}))$.
\begin{align*}
\lefteqn{(u-v)[D_{a;i,j}(u), E_{b;h,k}(v)]=}\\
&& \qquad\qquad\left\{
  \begin{array}{ll}
  (-1)^{\pa{b}}\big\lbrace\delta_{a,b}\delta_{h,j}D_{a;i,p}(u)\big(E_{a;p,k}(v)-E_{a;p,k}(u)\big)\\[2mm]
  \qquad\qquad\qquad -\delta_{a,b+1}D_{a;i,k}(u)\big(E_{b;h,j}(v)-E_{b;h,j}(u)\big)\big\rbrace, \;\;\text{if}\;\; b\ne m,\\[3mm]
   \delta_{a,b}\delta_{h,j}D_{a;i,p}(u)\big(E_{a;p,k}(v)-E_{a;p,k}(u)\big)\\[2mm]
   \qquad\qquad\qquad +\delta_{a,b+1}D_{a;i,k}(u)\big(E_{b;h,j}(v)-E_{b;h,j}(u)\big),\;\; \text{if}\;\; b = m,
        \end{array}\right.\notag\\[2mm]
\lefteqn{(u-v)[D_{a;i,j}(u), F_{b;h,k}(v)]=}\\
&& \qquad\qquad\left\{
  \begin{array}{ll}
 (-1)^{\pa{b}}\big\lbrace -\delta_{a,b}\delta_{k,i}\big(F_{b;h,p}(v)-F_{b;h,p}(u)\big)D_{a;p,j}(u)\\[2mm]
  \qquad\qquad\qquad +\delta_{a,b+1}\big(F_{b;i,k}(v)-F_{b;i,k}(u)\big)D_{a;h,j}(u) \big\rbrace, \;\;\text{if}\;\; b\ne m,\\[3mm]
 -\delta_{a,b}\delta_{k,i}\big(F_{b;h,p}(v)-F_{b;h,p}(u)\big)D_{a;p,j}(u)\\[2mm]
  \qquad\qquad\qquad -\delta_{a,b+1}\big(F_{b;i,k}(v)-F_{b;i,k}(u)\big)D_{a;h,j}(u),\text{ if }\;b = m,
   \end{array}\right.\notag
\end{align*}
\begin{eqnarray*}
\lefteqn{\hspace*{-10mm}(u-v)[E_{a;i,j}(u), E_{a;h,k}(v)]=}\\
&&\qquad\left\{
  \begin{array}{ll}
      (-1)^{\pa{a}} \big(E_{a;i,k}(u)-E_{a;i,k}(v)\big)\big(E_{a;h,j}(u)-E_{a;h,j}(v)\big),\;\;\text{if}\;\; a\ne m,\\[3mm]
       \big(E_{a;i,k}(u)-E_{a;i,k}(v)\big)\big(E_{a;h,j}(v)-E_{a;h,j}(u)\big),\;\;\text{if}\;\; a=m,
  \end{array}\right.\\[3mm]
\lefteqn{\hspace*{-10mm}(u-v)[F_{a;i,j}(u), F_{a;h,k}(v)]=}\\
&&\qquad\left\{
  \begin{array}{ll}
   -(-1)^{\pa{a}} \big(F_{a;i,k}(u)-F_{a;i,k}(v)\big)\big(F_{a;h,j}(u)-F_{a;h,j}(v)\big),\;\;\text{if}\;\; a\ne m,\\                                                                                                                                                                                  [3mm]
   \big(F_{a;i,k}(u)-F_{a;i,k}(v)\big)\big(F_{a;h,j}(v)-F_{a;h,j}(u)\big) ,\;\;\text{if}\;\; a=m,
  \end{array}\right.
\end{eqnarray*}

\begin{equation}\notag
(u-v)[E_{a;i,j}(u), F_{b;h,k}(v)]=
  \delta_{a,b}(-1)^{\pa{b+1}}\big(D^{\prime}_{a;i,k}(u) D_{a+1;h,j}(u)-D_{a+1;h,j}(v) D^{\prime}_{a;i,k}(v)\big),
\end{equation}
\begin{align*}\notag
&(u-v)\big[E_{a;i,j}(u), E_{a+1;h,k}(v)\big]=\delta_{h,j}(-1)^{\overline{a+1}}\big(E_{a;i,q}(u)E_{a+1;q,k}(v)\\
&\qquad\qquad\qquad\qquad\qquad\qquad
-E_{a;i,q}(v)E_{a+1;q,k}(v)+E_{a,a+2;i,k}(v)-E_{a,a+2;i,k}(u)\big),\\[2mm]\notag
&(u-v)\big[F_{a;i,j}(u), F_{a+1;h,k}(v)\big]=\delta_{i,k}(-1)^{\overline{a+1}}\big(-F_{a+1;h,q}(v)F_{a;q,j}(u)\\
&\qquad\qquad\qquad\qquad\qquad\qquad
+F_{a+1;h,q}(v)F_{a;q,j}(v)-F_{a+2,a;h,j}(v)+F_{a+2,a;h,j}(u)\big),
\end{align*}
\begin{eqnarray*}
(u-v)[E_{a;i,j}(u), E_{b;h,k}(v)]=0,&&
\quad\text{ if \;\; $b>a+1$ \;\;or\;\; if \;\; $b=a+1$ and $h \neq j$},\\[3mm]
(u-v)[F_{a;i,j}(u), F_{b;h,k}(v)]=0,&&
\quad\text{ if \;\; $b>a+1$ \;\;or\;\; if \;\; $b=a+1$ and $i \neq k$},
\end{eqnarray*}
\begin{eqnarray*}
\big[E_{a;i,j}(u),[E_{a;h,k}(v),E_{b;f,g}(w)]\big]+
\big[E_{a;i,j}(v),[E_{a;h,k}(u),E_{b;f,g}(w)]\big]=0,&\;\; |a-b|\geq 1,\\[3mm]
\big[F_{a;i,j}(u),[F_{a;h,k}(v),F_{b;f,g}(w)]\big]+
\big[F_{a;i,j}(v),[F_{a;h,k}(u),F_{b;f,g}(w)]\big]=0,&\;\; |a-b|\geq 1,
\end{eqnarray*}\label{erelation}
where $\pa{a}:=0$ if $1\leq a\leq m$ and $\pa{a}:=1$ if $m+1\leq a\leq m+n$.
\end{proposition}}
\begin{proof}
This is the consequence of Theorem \ref{bknonsuper}, Proposition \ref{1stsuper}, Lemmas \ref{mmn}$-$6.3, together with the maps $\psi_k$ and $\zeta_{M|N}$.
\end{proof}

The next lemma is a block generalization of \cite[Lemma 5]{Go} and the proof is essentially the same, except that we are using block decompositions. The relations are purely super phenomenons.
\begin{lemma}\label{extra}
Associated to $\mu=(\mu_1,\mu_2,\ldots\mu_m\,|\,\mu_{m+1},\ldots,\mu_{m+n})$ with $m>1$ and $n>1$, we have the following identities in $Y_{\mu}$.
\begin{equation}\label{eeee}
\big[\,[E_{m-1;i,j}^{(r)},E_{m;h,k}^{(1)}],[E_{m;h_0,k_0}^{(1)},E_{m+1;f,g}^{(s)}]\,\big]=0,
\end{equation}
\begin{equation}\label{ffff}
\big[\,[F_{m-1;i,j}^{(r)},F_{m;h,k}^{(1)}],[F_{m;h_0,k_0}^{(1)},F_{m+1;f,g}^{(s)}]\,\big]=0,
\end{equation}
for all admissible $f,g,h,i,j,k,h_0,k_0,r,s$.
\end{lemma}
\begin{proof}
By using the maps $\zeta_{M|N}$ and $\psi$, it is enough to show (\ref{eeee}) in the case $m=n=2$ only. Therefore, we want to show (7.1) in $Y_{(\mu_1,\mu_2|\mu_3,\mu_4)}$, i.e.,
\begin{equation}
\big[\,[E_{1;i,j}^{(r)},E_{2;h,k}^{(1)}]\,,\,[E_{2;h_0,k_0}^{(1)},E_{3;f,g}^{(s)}]\,\big]=0.
\end{equation}
We first claim that for all admissible $i,j,h,k,$
\begin{equation}\label{exeee}
[E_{1,3;i,j}(u)\,,\,E_{2;h,q}(v)E_{3;q,k}(v)-E_{2,4;h,k}(v)\,]=0,
\end{equation}
where the index $q$ is summed over $1,2,\ldots,\mu_3$.
To prove the claim, we use (\ref{t12}) and (\ref{t'12}) associated to the composition $(\mu_1,\mu_2\,|\,\mu_3,\mu_4)$ to derive the following identities.
\begin{eqnarray}
E_{1,3;i,j}(u)&=&D^{\prime}_{1;i,p}(u)t_{p,\mu_1+\mu_2+j}(u),\notag\\
E_{2;h,q}(v)E_{3;q,k}(v)-E_{2,4;h,k}(v)&=&t^{\prime}_{\mu_1+h,\mu_1+\mu_2+\mu_3+r}(v)D_{4;r,k}(v)\notag,
\end{eqnarray}
for all $1\le i\le\mu_1$, $1\le j\le\mu_3$, $1\le h\le\mu_2$, $1\le k\le\mu_4$, and the indices $p$, $q$, $r$ are summed over $\mu_1$, $\mu_3$, $\mu_4$, respectively. Substituting these identities into the bracket in (\ref{exeee}) and setting a notation $n_a:=\mu_1+\mu_2+\ldots+\mu_a$ for short, we have
\begin{align*}
[E&_{1,3;i,j}(u),E_{2;h,q}(v)E_{3;q,k}(v)-E_{2,4;h,k}(v)]\\
&=[D^{\prime}_{1;i,p}(u)t_{p,n_2+j}(u),t^{\prime}_{\mu_1+h,n_3+r}(v)D_{4;r,k}(v)]\\
&=D^{\prime}_{1;i,p}(u)t_{p,n_2+j}(u)t^{\prime}_{\mu_1+h,n_3+r}(v)D_{4;r,k}(v)
+t^{\prime}_{\mu_1+h,n_3+r}(v)D_{4;r,k}(v)D^{\prime}_{1;i,p}(u)t_{p,n_2+j}(u)\\
&=D^{\prime}_{1;i,p}(u)t_{p,n_2+j}(u)t^{\prime}_{\mu_1+h,n_3+r}(v)D_{4;r,k}(v)
+t^{\prime}_{\mu_1+h,n_3+r}(v)D^{\prime}_{1;i,p}(u)D_{4;r,k}(v)t_{p,n_2+j}(u)\\
&=D^{\prime}_{1;i,p}(u)t_{p,n_2+j}(u)t^{\prime}_{\mu_1+h,n_3+r}(v)D_{4;r,k}(v)
+D^{\prime}_{1;i,p}(u)t^{\prime}_{\mu_1+h,n_3+r}(v)t_{p,n_2+j}(u)D_{4;r,k}(v)\\
&=D^{\prime}_{1;i,p}(u)[t_{p,n_2+j}(u),
t^{\prime}_{\mu_1+h,n_3+r}(v)]D_{4;r,k}(v)=0,\text{ \; and the claim follows.}
\end{align*}
Note that in the above computation we have used the facts that
\[
D_{1;i,j}(u)=t_{ij}(u) \qquad \text{and} \qquad D^{\prime}_{4;i,j}(u)=t^{\prime}_{n_3+i,n_3+j}(u),
\]
therefore \qquad
$[D_{1;i,j}(u),t^{\prime}_{\mu_1+h,n_3+k}(v)]=0$\quad
and \quad $[D^{\prime}_{4;i,j}(u),t_{h,n_2+k}(v)]=0$ \;by (\ref{usefull}).

It suffices to prove (7.3) when $h=j$ and $k_0=f$, by Lemma \ref{mmn}($b$). Computing the following bracket by Lemma \ref{mmn}($b$), we have
\begin{align*}
(u-v&)(w-z)\big[\,[E_{1;i,j}(u),E_{2;j,k}(v)]\,,\,[E_{2;h_0,f}(w),E_{3;f,g}(z)]\,\big]\notag\\
&=\big[\,E_{1;i,q}(u)E_{2;q,k}(v)-E_{1;i,q}(v)E_{2;q,k}(v)+E_{1,3;i,k}(v)-E_{1,3;i,k}(u),\notag\\
&\; -E_{2;h_0,p}(w)E_{3;p,g}(z)+E_{2;h_0,p}(z)E_{3;p,g}(z)-E_{2,4;h_0,g}(z)+E_{2,4;h_0,g}(w)\,\big].\notag
\end{align*}
Taking its coefficient of $u^{-r}z^{-s}v^0w^0$, we have
\[
\sum_{t=1}^{s-1}[-E_{1,3;i,k}^{(r)},E_{2;h_0,p}^{(s-t)}E_{3;p,g}^{(t)}]+[-E_{1,3;i,k}^{(r)},-E_{2,4;h_0,g}^{(s)}],
\]
and it equals the coefficient of $u^{-r}z^{-s}$ in $[E_{1,3;i,k}(u),-E_{2;h_0,p}(z)E_{3;p,g}(z)+E_{2,4;h_0,g}(z)]$, which is zero by (7.4).

Finally, the coefficient of $u^{-r}z^{-s}v^0w^0$ in
\begin{equation*}
(u-v)(w-z)\big[\,[E_{1;i,j}(u),E_{2;j,k}(v)]\,,\,[E_{2;h_0,f}(w),E_{3;f,g}(z)]\,\big]
\end{equation*}
is exactly $-\big[\,[E_{1;i,j}^{(r)},E_{2;j,k}^{(1)}]\,,\,[E_{2;h_0,f}^{(1)},E_{3;f,g}^{(s)}]\,\big]$ and (7.3) follows.
\end{proof}

Recall the fact stated in Theorem 1 that $Y(\gl_{M|N})$ is generated as an algebra by the set $\big\lbrace D_{a;i,j}^{(r)}, D_{a;i,j}^{\prime(r)}, E_{a;i,j}^{(r)}, F_{a;i,j}^{(r)}\big\rbrace$. The following theorem describes the relations among these generators.

\begin{theorem}\label{srlns}
The following relations hold in $Y(\gl_{M|N})$ for all admissible indices $a,b,f,g$, $h,i,j,k,l,r,s,h_0,k_0$:
\begin{eqnarray}
D_{a;i,j}^{(0)}&=&\delta_{ij}\,,\\
\sum_{t=0}^{r}D_{a;i,p}^{(t)}D_{a;p,j}^{\prime (r-t)}&=&\delta_{r0}\delta_{ij}\,,\\
\big[D_{a;i,j}^{(r)},D_{b;h,k}^{(s)}\big]&=&
    \delta_{ab}\sum_{t=0}^{min(r,s)-1}\big(D_{a;h,j}^{(t)}D_{a;i,k}^{(r+s-1-t)}-D_{a;h,j}^{(r+s-1-t)}D_{a;i,k}^{(t)}\big),
\end{eqnarray}
\begin{align}\notag
&\lefteqn{[D_{a;i,j}^{(r)}, E_{b;h,k}^{(s)}]=}\\
&\;\left\{
  \begin{array}{ll}
     \displaystyle  (-1)^{\pa{b}} \big( \delta_{a,b}\delta_{h,j}\sum_{t=0}^{r-1} D_{a;i,p}^{(t)}E_{a;p,k}^{(r+s-1-t)}
         -\delta_{a,b+1}\sum_{t=0}^{r-1}D_{a;i,k}^{(t)}E_{b;h,j}^{(r+s-1-t)} \big), \;b \ne m,\\[4mm]
     \displaystyle    \delta_{a,b}\delta_{h,j}\sum_{t=0}^{r-1} D_{a;i,p}^{(t)}E_{a;p,k}^{(r+s-1-t)}
          +\delta_{a,b+1}\sum_{t=0}^{r-1}D_{a;i,k}^{(t)}E_{b;h,j}^{(r+s-1-t)}, \;b = m,
 \end{array}\right.\\[4mm]\notag
 \end{align}
 \begin{align}\notag
&\lefteqn{[D_{a;i,j}^{(r)}, F_{b;h,k}^{(s)}]=}\\
&\left\{
  \begin{array}{ll}
  \displaystyle (-1)^{\pa{b}} \big( -\delta_{a,b}\delta_{k,i}\sum_{t=0}^{r-1}
  F_{b;h,p}^{(r+s-1-t)}D_{a;p,j}^{(t)}
  +\delta_{a,b+1}\sum_{t=0}^{r-1}F_{b;i,k}^{(r+s-1-t)}D_{a;h,j}^{(t)} \big), \;b \ne m,\\[4mm]
   \displaystyle   -\delta_{a,b}\delta_{k,i}\sum_{t=0}^{r-1}F_{b;h,p}^{(r+s-1-t)}D_{a;p,j}^{(t)}
      -\delta_{a,b+1}\sum_{t=0}^{r-1}F_{b;i,k}^{(r+s-1-t)}D_{a;h,j}^{(t)}\,, \;b = m,
\end{array}\right.
\end{align}
\begin{equation}
[E_{a;i,j}^{(r)}, E_{a;h,k}^{(s)}]=
 \left\{
  \begin{array}{ll}
    \displaystyle      (-1)^{\pa{a}}\big(\sum_{t=1}^{s-1}E_{a;i,k}^{(t)}E_{a;h,j}^{(r+s-1-t)}
          -\sum_{t=1}^{r-1}E_{a;i,k}^{(t)}E_{a;h,j}^{(r+s-1-t)}\big),& a\neq m,\\[4mm]
     \displaystyle     \sum_{t=1}^{r-1}E_{a;i,k}^{(t)}E_{a;h,j}^{(r+s-1-t)}
          -\sum_{t=1}^{s-1}E_{a;i,k}^{(t)}E_{a;h,j}^{(r+s-1-t)} ,& a=m,\\[4mm]
  \end{array}\right.\\
\end{equation}
\begin{equation}
[F_{a;i,j}^{(r)}, F_{a;h,k}^{(s)}]=
\left\{
  \begin{array}{ll}
   \displaystyle (-1)^{\pa{a}}\big(\sum_{t=1}^{r-1}F_{a;i,k}^{(r+s-1-t)}F_{a;h,j}^{(t)}
   -\sum_{t=1}^{s-1}F_{a;i,k}^{(r+s-1-t)}F_{a;h,j}^{(t)}\big),& a\neq m,\\
   \displaystyle \sum_{t=1}^{r-1}F_{a;i,k}^{(r+s-1-t)}F_{a;h,j}^{(t)}-
   \sum_{t=1}^{s-1}F_{a;i,k}^{(r+s-1-t)}F_{a;h,j}^{(t)} \,,& a=m,\\
  \end{array}\right.\\
\end{equation}
\begin{equation}
[E_{a;i,j}^{(r)}, F_{b;h,k}^{(s)}]=
         -(-1)^{\pa{b+1}}\delta_{a,b} \sum_{t=0}^{r+s-1}D_{a+1;h,j}^{(r+s-1-t)} D^{\prime (t)}_{a;i,k}\,,
\end{equation}
\begin{align}
&[E_{a;i,j}^{(r+1)}, E_{a+1;h,k}^{(s)}]-[E_{a;i,j}^{(r)}, E_{a+1;h,k}^{(s+1)}]
=(-1)^{\pa{a+1}}\delta_{h,j}E_{a;i,q}^{(r)}E_{a+1;q,k}^{(s)}\,,\\[3mm]
&[F_{a;i,j}^{(r+1)}, F_{a+1;h,k}^{(s)}]-[F_{a;i,j}^{(r)}, F_{a+1;h,k}^{(s+1)}]
=-(-1)^{\pa{a+1}}\delta_{i,k}F_{a+1;h,q}^{(s)}F_{a;q,j}^{(r)}\,,\\[3mm]
&[E_{a;i,j}^{(r)}, E_{b;h,k}^{(s)}] = 0
\quad\text{\;\;if\;\; $b>a+1$ \;\;or\;\; \;if\;\;$b=a+1$ and $h \neq j$},\\[3mm]
&[F_{a;i,j}^{(r)}, F_{b;h,k}^{(s)}] = 0
\quad\text{\;\;if\;\; $b>a+1$ \;\;or\;\; \;if\;\;$b=a+1$ and $i \neq k$},
\end{align}
\begin{eqnarray}
&\big[E_{a;i,j}^{(r)},[E_{a;h,k}^{(s)},E_{b;f,g}^{(l)}]\big]+
\big[E_{a;i,j}^{(s)},[E_{a;h,k}^{(r)},E_{b;f,g}^{(l)}]\big]=0, &|a-b|\geq 1,\\[3mm]
&\big[F_{a;i,j}^{(r)},[F_{a;h,k}^{(s)},F_{b;f,g}^{(l)}]\big]+
\big[F_{a;i,j}^{(s)},[F_{a;h,k}^{(r)},F_{b;f,g}^{(l)}]\big]=0, &|a-b|\geq 1,\\[3mm]
&\big[\,[E_{m-1;i,j}^{(r)},E_{m;h,k}^{(1)}]\,,\,[E_{m;h_0,k_0}^{(1)},E_{m+1;f,g}^{(s)}]\,\big]=0, &\text{when\;\;}  m>1 ,n>1,\\[3mm]
&\big[\,[F_{m-1;i,j}^{(r)},F_{m;h,k}^{(1)}]\,,\,[F_{m;h_0,k_0}^{(1)},F_{m+1;f,g}^{(s)}]\,\big]=0, &\text{when\;\;}m>1, n>1,
\end{eqnarray}
where $\pa{a}:=0$ if $1\leq a\leq m$, $\pa{a}:=1$ if $m+1\leq a\leq m+n$, and the index p (resp. q) is summed over 1,$\ldots,\mu_1$ (resp. $1,\ldots,\mu_2$)
\end{theorem}
\begin{proof}
(7.5)$-$(7.7) follow from Proposition \ref{MNdd0}, while the others come from {Proposition}~\ref{pgpropo}, Lemma \ref{extra} and the identity
\[
\dfrac{S(v)-S(u)}{u-v}=\sum_{r,s\ge 1}S^{(r+s-1)}u^{-r}v^{-s},
 \]
for any formal series $S(u)=\sum_{r\ge 0}S^{(r)}u^{-r}$.
\end{proof}
\begin{remark}
In the special case where all $\mu_i=1$, the right hand side of (7.10) and (7.11) degenerate to zero when a=m. See \cite[Theorem 3]{Go}.
\end{remark}

In fact, the relations in Theorem~\ref{srlns} are enough as defining relations of the super Yangian $Y(\gl_{M|N})$.
\begin{theorem}\label{Pg} 
The super Yangian $Y(\gl_{M|N})$ is generated by the elements
\begin{align*}
&\lbrace D_{a;i,j}^{(r)}, D_{a;i,j}^{\prime(r)} \,|\, 1\leq a\leq m+n, 1\leq i,j\leq \mu_a, r\geq 0\rbrace,\\
&\lbrace E_{a;i,j}^{(r)} \,|\, 1\leq a< m+n, 1\leq i\leq \mu_a, 1\leq j\leq\mu_{a+1}, r\geq 1\rbrace,\\
&\lbrace F_{a;i,j}^{(r)} \,|\, 1\leq a< m+n, 1\leq i\leq\mu_{a+1}, 1\leq j\leq \mu_a, r\geq 1\rbrace,
\end{align*}
subject to the relations (7.5)$-$(7.20).
\end{theorem}
\begin{proof}
Recall the notation $Y_{\mu}:=Y(\gl_{M|N})$ defined in Section 6. Let $\widehat{Y}_{\mu}$ denote the abstract algebra generated by the elements and relations as in the statement of Theorem~\ref{Pg}. We may further define all the other $E_{a,b;i,j}^{(r)}$ and $F_{b,a;i,j}^{(r)}$ in $\widehat{Y}_{\mu}$ by the relations (3.18), and it is not hard to show that this definition is independent of the choices of $k$ \cite[p.22]{BK1}. Let $\Gamma$ be the map
\[
\Gamma: \widehat{Y}_{\mu}\longrightarrow Y_{\mu}
\]
sending every element in $\widehat{Y}_{\mu}$ into the element in $Y_{\mu}$ with the same name. By Theorem~\ref{gendef} and Theorem~\ref{srlns}, the map $\Gamma$ is a surjective algebra homomorphism. Therefore, it remains to prove that $\Gamma$ is also injective. The injectivity will be proved in Section 8.
\end{proof}

\section{Injectivity of $\Gamma$}
Our strategy of proving the injectivity of $\Gamma$ is as follows: we find a spanning set for $\widehat{Y}_{\mu}$ (see Proposition~\ref{ind1}) and show that the images of the spanning set for $\widehat{Y}_{\mu}$ under $\Gamma$ is linearly independent in $Y_{\mu}$ (see Proposition~\ref{ind2}).
\begin{proposition}\label{ind1}
$\widehat{Y}_{\mu}$ is spanned as a vector space by the monomials in the elements $\lbrace D_{a;i,j}^{(r)}, E_{a,b;i,j}^{(r)}, F_{b,a;i,j}^{(r)}\rbrace$ taken in certain fixed order.
\end{proposition}

\begin{proof}
Let $\widehat{Y}^0_{\mu}$ (resp. $\widehat{Y}^+_{\mu}$, $\widehat{Y}^-_{\mu}$) denote the subalgebras of
$\widehat{Y}_{\mu}$ generated by the elements $\lbrace D_{a;i,j}^{(r)}\rbrace$ (resp. $\lbrace E_{a,b;i,j}^{(r)}\rbrace$, $\lbrace F_{b,a;i,j}^{(r)}\rbrace$). By the relations in Theorem \ref{srlns}, $\widehat{Y}_{\mu}$ is spanned by the monomials where all $F$'s come before all $D$'s and all $D$'s come before all $E$'s.

Define a filtration on $\widehat{Y}_{\mu}$ by setting
\begin{equation*}
\text{deg}(D_{a;i,j}^{(r)})=\text{deg}(E_{a,b;i,j}^{(r)})=\text{deg}(F_{b,a;i,j}^{(r)})=r-1,\qquad\text{for all}\;\; r\ge 1,
\end{equation*}
and denote the associated graded algebra by $gr^L\widehat{Y}_{\mu}$. The above argument implies that the multiplication map is surjective,
\[
gr^L\widehat{Y}^-_{\mu}\otimes gr^L\widehat{Y}^0_{\mu}\otimes gr^L\widehat{Y}^+_{\mu}\twoheadrightarrow gr^L\widehat{Y}_{\mu}.
\]
Moreover, $gr^L\widehat{Y}^0_{\mu}$ is commutative by Proposition \ref{MNdd0}. It follows that $\widehat{Y}^0_{\mu}$ is spanned by the monomials in $\lbrace D_{a;i,j}^{(r)}\rbrace$ in certain fixed order. Hence it is enough to show that $gr^L\widehat{Y}^+_{\mu}$ is spanned by the monomials in $E$'s in certain order, and the swap map $\zeta_{N|M}$ will show that $gr^L\widehat{Y}^-_{\mu}$ is spanned by the monomials in $F$'s in certain order. 

We denote the image of $E_{a,b;i,j}^{(r)}$ in the graded algebra $gr_{r-1}^L\widehat{Y}_{\mu}^+$ by $\pa{E}_{a,b;i,j}^{(r)}$. We have the following.\\[2mm]
{\bf{Claim*:}}
For all admissible $a,b,c,d,i,j,h,k,r,s$, we have
\begin{equation}\label{ind3}
[\pa{E}_{a,b;i,j}^{(r)},\pa{E}_{c,d;h,k}^{(s)}]=(-1)^{\pa{b}}\delta_{b,c}\delta_{h,j}\pa{E}_{a,d;i,k}^{(r+s-1)}
-(-1)^{\pa{a}\,\pa{b}+\pa{a}\,\pa{c}+\pa{b}\,\pa{c}}\delta_{a,d}\delta_{i,k}\pa{E}_{c,b;h,j}^{(r+s-1)}.
\end{equation}

Assuming the claim, we have that the graded algebra $gr^L\widehat{Y}_{\mu}^+$ is spanned by the monomials in $\lbrace\pa{E}_{a,b;i,j}^{(r)}\rbrace$ in certain order and hence $\widehat{Y}_{\mu}^+$ is spanned by the monomials in $\lbrace E_{a,b;i,j}^{(r)}\rbrace$ in certain order as well and therefore Proposition~\ref{ind1} is established.
\end{proof}

To establish the claim*, we first prove some special cases.
\begin{lemma} The following identities hold in \text{gr}$^L\widehat{Y}_{\mu}^+$:
 \begin{enumerate}
  \item[(a)] \begin{equation}
              [\pa{E}_{a,a+1;i,j}^{(r)},\pa{E}_{b,b+1;h,k}^{(s)}]=0,\;\text{if}\;|a-b|\ne 1,
             \end{equation}
  \item[(b)] \begin{equation}
              [\pa{E}_{a,a+1;i,j}^{(r)},\pa{E}_{b,b+1;h,k}^{(s)}]=
              [\pa{E}_{a,a+1;i,j}^{(r-1)},\pa{E}_{b,b+1;h,k}^{(s+1)}],\; \text{if}\; |a-b|=1,
             \end{equation}
  \item[(c)] \begin{equation}
              \big[\pa{E}_{a,a+1;i,j}^{(r)},[\pa{E}_{a,a+1;h,k}^{(s)},\pa{E}_{b,b+1;f,g}^{(t)}]\big]=
              -\big[\pa{E}_{a,a+1;i,j}^{(s)},[\pa{E}_{a,a+1;h,k}^{(r)},\pa{E}_{b,b+1;f,g}^{(t)}]\big],
             \end{equation}
             \; \text{if}\; $|a-b|=1$,
  \item[(d)] \begin{equation}
              \pa{E}_{a,b;i,j}^{(r)}=(-1)^{\pa{b-1}}[\pa{E}_{a,b-1;i,h}^{(r)},\pa{E}_{b-1,b;h,j}^{(1)}]
              =(-1)^{\pa{a+1}}[\pa{E}_{a,a+1;i,k}^{(1)},\pa{E}_{a+1,b;k,j}^{(r)}],
              \end{equation}
              for all $b>a+1$ and any $1\leq h\leq\mu_{b-1}$, $1\leq k\leq\mu_{a+1}$.
 \end{enumerate}
\end{lemma}

\begin{proof}
(8.2) and (8.3) follow from (7.15) and (7.13). (8.4) follows from (7.17) and (8.5) follows from (\ref{ter}).
\end{proof}

\begin{lemma}
The following identities hold in gr$^L\widehat{Y}_{\mu}^+$:
 \begin{enumerate}
  \item[(a)]
    \begin{equation}
     [\pa{E}_{a,a+2;i,j}^{(r)},\pa{E}_{a+1,a+2;h,k}^{(s)}]=0,\;\; \text{for all}\;\; 1\leq a\leq m+n-2,
    \end{equation}
  \item[(b)]
    \begin{equation}
     [\pa{E}_{a,a+1;i,j}^{(r)},\pa{E}_{a,a+2;h,k}^{(s)}]=0,\;\; \text{for all}\;\; 1\leq a\leq m+n-2,
    \end{equation}
  \item[(c)]
    \begin{equation}
     [\pa{E}_{a,a+2;i,j}^{(r)},\pa{E}_{a+1,a+3;h,k}^{(s)}]=0,\;\; \text{for all}\;\; 1\leq a\leq m+n-3,
    \end{equation}
  \item[(d)]
    \begin{equation}
     [\pa{E}_{a,b;i,j}^{(r)},\pa{E}_{c,c+1;h,k}^{(s)}]=0,\;\; \text{for all}\;\; 1\leq a<c<b\leq m+n.
    \end{equation}
 \end{enumerate}
\end{lemma}

\begin{proof}
(a) By (8.5) and (8.4), we have
 \begin{align*}
 (-1)^{\pa{a+1}}[\pa{E}^{(r)}_{a,a+2;i,j}\,,\,\pa{E}^{(s)}_{a+1,a+2;h,k}]=&
  \big[\,[\pa{E}^{(r)}_{a,a+1;i,f},\pa{E}^{(1)}_{a+1,a+2;f,j}]\,,\,\pa{E}^{(s)}_{a+1,a+2;h,k}\big]\\
 =&-\big[\,[\pa{E}^{(r)}_{a,a+1;i,f},\pa{E}^{(s)}_{a+1,a+2;f,j}]\,,\,\pa{E}^{(1)}_{a+1,a+2;h,k}\,\big]\\
 =&-\big[\,[\pa{E}^{(r+s-1)}_{a,a+1;i,f},\pa{E}^{(1)}_{a+1,a+2;f,j}]\,,\,\pa{E}^{(1)}_{a+1,a+2;h,k}\,\big]
 \end{align*}
 and the last term is zero by (8.4).

(b) The same method in (a) works, except that we apply (8.5) on the term $\pa{E}_{a,a+2;h,k}^{(s)}$.

(c) It takes some effort in this case due to the $\mathbb{Z}_2$-grading. First assume that $a\neq~{m-1}$. We apply (8.5) on the left hand side of (8.8) and use the super-Jacobi identity:
\begin{align*}
 [\pa{E}_{a,a+2;i,j}^{(r)},\pa{E}_{a+1,a+3;h,k}^{(s)}]&=
 (-1)^{\pa{a+1}+\pa{a+2}}\big[\,[\pa{E}_{a,a+1;i,h}^{(r)},\pa{E}_{a+1,a+2;h,j}^{(1)}]
 \,,\,[\pa{E}_{a+1,a+2;h,j}^{(1)},\pa{E}_{a+2,a+3;j,k}^{(s)}]\,\big]\\
&=(-1)^{\pa{a+1}+\pa{a+2}}\Big[\,\big[\,[\pa{E}_{a,a+1;i,h}^{(r)},\pa{E}_{a+1,a+2;h,j}^{(1)}],
 \pa{E}_{a+1,a+2;h,j}^{(1)}\big],\pa{E}_{a+2,a+3;j,k}^{(s)}\Big]\\
&+\varepsilon(-1)^{\pa{a+1}+\pa{a+2}}\Big[\pa{E}_{a+1,a+2;h,j}^{(1)},
 \big[\,[\pa{E}_{a,a+1;i,h}^{(r)},\pa{E}_{a+1,a+2;h,j}^{(1)}],\pa{E}_{a+2,a+3;j,k}^{(s)}\big]\,\Big],
\end{align*}
 where $\varepsilon$ is $(-1)^{\pa{\alpha} \pa{\beta}}$, $\pa{\alpha}$ is the degree of $[\pa{E}_{a,a+1;i,h}^{(r)},\pa{E}_{a+1,a+2;h,j}^{(1)}]$ and $\pa{\beta}$ is the degree of $\pa{E}_{a+1,a+2;h,j}^{(1)}$. By (8.4), the first term is zero. Moreover, by our assumption that $a\neq m-1$, the elements $\pa{E}_{a+1,a+2;h,j}^{(1)}$ is even and hence $\varepsilon$ is 1. Keep using the super-Jacobi identity and Lemma 8.2, we may deduce that the above equals to
\begin{align*}
&(-1)^{\pa{a+1}+\pa{a+2}}\Big[\pa{E}_{a+1,a+2;h,j}^{(1)},
 \big[\,[\pa{E}_{a,a+1;i,h}^{(r)},\pa{E}_{a+1,a+2;h,j}^{(1)}],\pa{E}_{a+2,a+3;j,k}^{(s)}\big]\,\Big]\\
=&(-1)^{\pa{a+1}+\pa{a+2}}\Big[\pa{E}_{a+1,a+2;h,j}^{(1)},
 \big[\pa{E}_{a,a+1;i,h}^{(r)},[\pa{E}_{a+1,a+2;h,j}^{(1)},\pa{E}_{a+2,a+3;j,k}^{(s)}]\,\big]\,\Big]+0\\
=&(-1)^{\pa{a+1}+\pa{a+2}}\big[\,[\pa{E}_{a+1,a+2;h,j}^{(1)},\pa{E}_{a,a+1;i,h}^{(r)}],
[\pa{E}_{a+1,a+2;h,j}^{(1)},\pa{E}_{a+2,a+3;j,k}^{(s)}]\,\big]+0\\
=&-(-1)^{\pa{a+1}+\pa{a+2}}\big[\,[\pa{E}_{a,a+1;i,h}^{(r)},\pa{E}_{a+1,a+2;h,j}^{(1)}]
,[\pa{E}_{a+1,a+2;h,j}^{(1)},\pa{E}_{a+2,a+3;j,k}^{(s)}]\,\big]\\
=&-[\pa{E}_{a,a+2;i,j}^{(r)},\pa{E}_{a+1,a+3;h,k}^{(s)}].
\end{align*}
Therefore, (8.8) is true for all $a\neq m-1$.

Now let $a=m-1$, same method shows that
\begin{align*}
&[\,\pa{E}_{m-1,m+1;i,j}^{(r)},\pa{E}_{m,m+2;h,k}^{(s)}\,]\\
=&\big[\,(-1)^{\pa{m}}\,[\,\pa{E}_{m-1,m;i,f}^{(r)}\,,\,\pa{E}_{m,m+1;f,j}^{(1)}\,]\,,\,
(-1)^{\pa{m+1}}\,[\,\pa{E}_{m,m+1;h,g}^{(1)}\,,\,\pa{E}_{m+1,m+2;g,k}^{(s)}\,]\,\big]\\
=&\pm \big[\,[\pa{E}_{m-1,m;i,f}^{(r)}\,,\,\pa{E}_{m,m+1;f,j}^{(1)}]\,,\,
[\pa{E}_{m,m+1;h,g}^{(1)}\,,\,\pa{E}_{m+1,m+2;g,k}^{(s)}]\,\big],
\end{align*}
which is zero by (\ref{eeee}) and hence (8.8) is true when $a=m-1$ as well.

(d) By super-Jacobi identity and (8.5), it is enough to show the following 2 cases:
\begin{equation}
[\pa{E}^{(r)}_{a,c+1;i,j},\pa{E}^{(s)}_{c,c+1;h,k}]=0,\;  \text{for all} \; a<c,
\end{equation}
and
\begin{equation}
[\pa{E}^{(r)}_{a,c+1;i,j},\pa{E}^{(s)}_{c,c+2;h,k}]=0,\;  \text{for all} \; a<c.
\end{equation}
They can be proved by using (8.2)$-$(8.8) and induction on $c-a$. We show (8.10) in detail here. When $c=a+1$, it follows directly from (8.6). Now assume $c>a+1$. By (8.5) and super-Jacobi identity, we have
\begin{multline*}
[\pa{E}^{(r)}_{a,c+1;i,j},\pa{E}^{(s)}_{c,c+1;h,k}]=
\big[\,(-1)^{\pa{a+1}}\,[\pa{E}^{(1)}_{a,a+1;i,f}\,,\,\pa{E}^{(r)}_{a+1,c+1;f,j}\,]\,,\,\pa{E}^{(s)}_{c,c+1;h,k}\,\big]\\
=(-1)^{\pa{a+1}}\big[\pa{E}^{(1)}_{a,a+1;i,f},[\pa{E}^{(r)}_{a+1,c+1;f,j},\pa{E}^{(s)}_{c,c+1;h,k}]\,\big]
\pm \big[\pa{E}^{(r)}_{a+1,c+1;f,j},[\pa{E}^{(1)}_{a,a+1;i,f},\pa{E}^{(s)}_{c,c+1;h,k}]\,\big].
\end{multline*}
The first term is zero by induction hypothesis and the second term is also zero by (8.2).
\end{proof}

\begin{proof}[Proof of claim*]
 Without loss of generality, we may assume that $a\leq c$. The proof is split into 7 cases and we prove them one by one.
\begin{description}
\item[Case 1.] $a<b<c<d$:\\
It follows directly from (8.2) and (8.5) that the bracket in (8.1) is zero.

\item[Case 2.] $a<b=c<d$:\\
By (8.3) and (8.5), we have
\begin{equation}
[\pa{E}^{(r+1)}_{b-1,b;i_1,j}\,,\,\pa{E}^{(s+1)}_{b,b+1;h,k_1}]
=[\pa{E}^{(r+s+1)}_{b-1,b;i_1,j}\,,\,\pa{E}^{(1)}_{b,b+1;h,k_1}]
=\delta_{h,j}(-1)^{\pa{b}}\pa{E}^{(r+s+1)}_{b-1,b+1;i_1,k_1}.
\end{equation}
Note that when $h\neq j$, the bracket is zero by (7.13) and hence the $\delta_{h,j}$ comes out.
Taking the bracket on both sides of the equation (8.12) with the elements
\[
\pa{E}^{(1)}_{b+1,b+2;k_1,k_2}, \pa{E}^{(1)}_{b+2,b+3;k_2,k_3}, \cdots , \pa{E}^{(1)}_{d-1,d;k_{d-1},k}
\]
from the right and using the super-Jacobi identity, (8.2) and (8.5), we have
\begin{equation}
[\pa{E}^{(r+1)}_{b-1,b;i_1,j}\,,\,\pa{E}^{(s+1)}_{b,d;h,k}]=\delta_{h,j}(-1)^{\pa{b}}\pa{E}^{(r+s+1)}_{b-1,d;i_1,k}\,.
\end{equation}
Taking brackets on both sides of (8.13) with the elements
\[
\pa{E}^{(1)}_{b-2,b-1;i_2,i_1}, \pa{E}^{(1)}_{b-3,b-2;i_3,i_2},\cdots,\pa{E}^{(1)}_{a,a+1;i,i_{b-a-1}}
\]
from the left and using exactly the same method as above, we have
\[
[\pa{E}^{(r)}_{a,b;i,j}\,,\,\pa{E}^{(s)}_{b,d;h,k}]=\delta_{h,j}(-1)^{\pa{b}}\pa{E}^{(r+s-1)}_{a,d;i,k},
\;\text{as desired}.
\]
\item[Case 3.] $a<c<b=d$:\\
Using the super-Jacobi identity, (8.5) and (8.9), we have
\begin{align*}
[\pa{E}^{(r)}_{a,b;i,j}&,\pa{E}^{(s)}_{c,b;h,k}]
=\big[\pa{E}^{(r)}_{a,b;i,j},(-1)^{\pa{c+1}}[\pa{E}^{(1)}_{c,c+1;h,f_1},\pa{E}^{(s)}_{c+1,b;f_1,k}]\,\big]\\
&=(-1)^{\pa{c+1}}\big[\,[\pa{E}^{(r)}_{a,b;i,j}\,,\,\pa{E}^{(1)}_{c,c+1;h,f_1}],\pa{E}^{(s)}_{c+1,b;f_1,k}\big]\\
&\quad\pm(-1)^{\pa{c+1}}\big[\pa{E}^{(1)}_{c,c+1;h,f_1}\,,\,[\pa{E}^{(r)}_{a,b;i,j}\,,\,\pa{E}^{(s)}_{c+1,b;f_1,k}]\,\big]\\
&=0\pm (-1)^{\pa{c+1}}\big[\pa{E}^{(1)}_{c,c+1;h,f_1}\,,\,[\pa{E}^{(r)}_{a,b;i,j}\,,\,\pa{E}^{(s)}_{c+1,b;f_1,k}]\,\big]\\
&=\cdots =\pm \Big[\pa{E}^{(1)}_{c,c+1;h,f_1}\,,\,[\pa{E}^{(1)}_{c+1,c+2;f_1,f_2},\ldots,
[\pa{E}^{(r)}_{a,b;i,j}\,,\,\pa{E}^{(s)}_{b-1,b;f_{b-1-c},k}]\,\big]\cdots\Big].
\end{align*}
By (8.9) again, the bracket $[\pa{E}^{(r)}_{a,b;i,j}\,,\,\pa{E}^{(s)}_{b-1,b;f_{b-1-c},k}]=0$.

\item[Case 4.] $a<c<d<b$:\\
Using the same method as in Case 3, we have
\begin{align*}
[\pa{E}^{(r)}_{a,b;i,j}&,\pa{E}^{(s)}_{c,d;h,k}]
=\big[\pa{E}^{(r)}_{a,b;i,j},(-1)^{\pa{c+1}}[\pa{E}^{(1)}_{c,c+1;h,f_1},\pa{E}^{(s)}_{c+1,d;f_1,k}]\,\big]\\
&=(-1)^{\pa{c+1}}\big[\,[\pa{E}^{(r)}_{a,b;i,j}\,,\,\pa{E}^{(1)}_{c,c+1;h,f_1}],\pa{E}^{(s)}_{c+1,d;f_1,k}]\,\big]\\
&\quad\pm(-1)^{\pa{c+1}}\big[\pa{E}^{(1)}_{c,c+1;h,f_1}\,,\,[\pa{E}^{(r)}_{a,b;i,j}\,,\,\pa{E}^{(s)}_{c+1,d;f_1,k}]\,\big]\\
&=0\pm(-1)^{\pa{c+1}}\big[\pa{E}^{(1)}_{c,c+1;h,f_1}\,,\,[\pa{E}^{(r)}_{a,b;i,j}\,,\,\pa{E}^{(s)}_{c+1,d;f_1,k}]\,\big]\\
&=\cdots =\pm \Big[\pa{E}^{(1)}_{c,c+1;h,f_1}\,,\,\big[\pa{E}^{(1)}_{c+1,c+2;f_1,f_2}\,,\ldots,[\pa{E}^{(r)}_{a,b;i,j}\,,\,\pa{E}^{(s)}_{d-1,d;f_{d-1-c},k}]\,\big]\cdots\Big].
\end{align*}
By (8.9) again, the bracket $[\pa{E}^{(r)}_{a,b;i,j}\,,\,\pa{E}^{(s)}_{d-1,d;f_{d-1-c},k}]=0$.

\item[Case 5.] $a<c<b<d$:\\
 We prove this case by induction on $d-b\geq 1$. When $d-b=1$, we have
 \begin{multline*}
  [\pa{E}^{(r)}_{a,b;i,j},\pa{E}^{(s)}_{c,b+1;h,k}]
  =\big[\pa{E}^{(r)}_{a,b;i,j},(-1)^{\pa{b}}[\pa{E}^{(s)}_{c,b;h,j},\pa{E}^{(1)}_{b,b+1;j,k}]\,\big]\\
  =(-1)^{\pa{b}}\big[\,[\pa{E}^{(r)}_{a,b;i,j}\,,\,\pa{E}^{(s)}_{c,b;h,j}]\,,\,\pa{E}^{(1)}_{b,b+1;j,k}\big]
  \pm (-1)^{\pa{b}}\big[\pa{E}^{(s)}_{c,b;h,j}\,,\,[\pa{E}^{(r)}_{a,b;i,j}\,,\,\pa{E}^{(1)}_{b,b+1;j,k}]\,\big].
 \end{multline*}
Now the bracket in the first term is zero by Case 3, and we may rewrite the whole second term as
$\pm[\pa{E}^{(r)}_{a,b+1;i,k},\pa{E}^{(s)}_{c,b;h,j}]$, which is zero by Case 4.
Assume that $d-b>1$, then $d-1>b$. By (8.5), the bracket becomes
\begin{align*}
[&\pa{E}^{(r)}_{a,b;i,j},\pa{E}^{(s)}_{c,d;h,k}]
=\big[\pa{E}^{(r)}_{a,b;i,j}\,,\,(-1)^{\pa{d-1}}[\pa{E}^{(s)}_{c,d-1;h,f}\,,\,\pa{E}^{(1)}_{d-1,d;f,k}]\,\big]\\
&=(-1)^{\pa{d-1}}\big[\,[\pa{E}^{(r)}_{a,b;i,j}\,,\,\pa{E}^{(s)}_{c,d-1;h,f}]\,,\,\pa{E}^{(1)}_{d-1,d;f,k}\big]
\pm\big[\pa{E}^{(s)}_{c,d-1;h,f}\,,\,[\pa{E}^{(r)}_{a,b;i,j}\,,\,\pa{E}^{(1)}_{d-1,d;f,k}]\,\big].
\end{align*}
The bracket in the first term is zero by induction hypothesis, while the bracket in the second term is zero as well by Case 1.

\item[Case 6.] $a=c<b<d$:
 \begin{align*}
 [\pa{E}^{(r)}_{a,b;i,j},\pa{E}^{(s)}_{a,d;h,k}]
 &=\big[\pa{E}^{(r)}_{a,b;i,j}\,,\,(-1)^{\pa{a+1}}[\pa{E}^{(1)}_{a,a+1;h,f}\,,\,\pa{E}^{(s)}_{a+1,d;f,k}]\,\big]\\
 &=(-1)^{\pa{a+1}}\big[\,[\pa{E}^{(r)}_{a,b;i,j}\,,\,\pa{E}^{(1)}_{a,a+1;h,f}]\,,\,\pa{E}^{(s)}_{a+1,d;h,k}\big]\\
 &\quad \pm\big[\pa{E}^{(1)}_{a,a+1;h,f}\,,\,[\pa{E}^{(r)}_{a,b;i,j}\,,\,\pa{E}^{(s)}_{a+1,d;f,k}]\,\big].
 \end{align*}
 Note that $[\pa{E}^{(r)}_{a,b;i,j}\,,\,\pa{E}^{(s)}_{a+1,d;f,k}]=0$ by Case 5. Hence it is enough to show that
 \begin{equation}
 [\pa{E}^{(r)}_{a,b;i,j}\,,\,\pa{E}^{(1)}_{a,a+1;h,f}]=0, \qquad \text{for all}\quad b>a.
 \end{equation}
 We prove (8.14) by induction on $b-a\geq 1$. When $b-a=1$, it follows from (8.2). Now
 assume $b-a>1$. By (8.5), we have
 \begin{align*}
 [\,\pa{E}^{(r)}_{a,b;i,j}\,,\,\pa{E}^{(1)}_{a,a+1;h,f}\,]
 &= \big[\,(-1)^{\pa{b-1}}\,[\,\pa{E}^{(r)}_{a,b-1;i,g}\,,\,\pa{E}^{(1)}_{b-1,b;g,j}\,]\,,\,\pa{E}^{(1)}_{a,a+1;h,f}\big]\\
 &=(-1)^{\pa{b-1}}\,\big[\pa{E}^{(r)}_{a,b-1;i,g}\,,\,[\pa{E}^{(1)}_{b-1,b;g,j}\,,\,\pa{E}^{(1)}_{a,a+1;h,f}]\,\big]\\
 &\quad \pm(-1)^{\pa{b-1}}\,\big[\pa{E}^{(1)}_{b-1,b;g,j}\,,\,[\pa{E}^{(r)}_{a,b-1;i,g}\,,\,\pa{E}^{(1)}_{a,a+1;h,f}]\,\big].
 \end{align*}

 Note that $[\pa{E}^{(r)}_{a,b-1;i,g}\,,\,\pa{E}^{(1)}_{a,a+1;h,f}]=0$ by induction hypothesis. Also by (8.2),
 $[\pa{E}^{(1)}_{b-1,b;g,j}\,,\,\pa{E}^{(1)}_{a,a+1;h,f}]=0$ unless $b-1=a+1$.
 When $b-1=a+1$, (8.14) becomes $[\pa{E}^{(r)}_{a,a+2;i,j}\,,\,\pa{E}^{(1)}_{a,a+1;h,f}]$, which is zero by (8.7).

\item[Case 7.] $a=c<b=d$:\\
 We claim that
 \begin{equation}
 [\pa{E}^{(r)}_{a,b;i,j}\,,\,\pa{E}^{(s)}_{a,b;h,k}]=0.
 \end{equation}
 If $b=a+1$, it follows directly from (8.2). If $b>a+1$, we may expand one term in the bracket of (8.15) by (8.5) as follow.
 \begin{align*}
  [\,\pa{E}^{(r)}_{a,b;i,j}\,,\,\pa{E}^{(s)}_{a,b;h,k}\,]
  &=\big[\,(-1)^{\pa{b-1}}\,[\pa{E}^{(r)}_{a,b-1;i,f}\,,\,\pa{E}^{(1)}_{b-1,b;f,j}]\,,\,\pa{E}^{(s)}_{a,b;h,k}\big]\\
  &=(-1)^{\pa{b-1}}\,\big[\pa{E}^{(r)}_{a,b-1;i,f}\,,\,[\pa{E}^{(1)}_{b-1,b;f,j}\,,\,\pa{E}^{(s)}_{a,b;h,k}]\,\big]\\
  &\quad\pm(-1)^{\pa{b-1}}\,\big[\pa{E}^{(1)}_{b-1,b;f,j}\,,\,[\pa{E}^{(r)}_{a,b-1;i,f}\,,\,\pa{E}^{(s)}_{a,b;h,k}]\,\big].
 \end{align*}
Note that $[\pa{E}^{(1)}_{b-1,b;f,j}\,,\,\pa{E}^{(s)}_{a,b;h,k}]=0$ by Case 3 and $[\pa{E}^{(r)}_{a,b-1;i,f}\,,\,\pa{E}^{(s)}_{a,b;h,k}]=0$ by Case~6. Therefore, we have proved (8.15).
\end{description}
This completes the proof of claim*.
\end{proof}

\begin{proposition}\label{ind2}
The images of the monomials in Proposition~\ref{ind1} under $\Gamma$ are linearly independent.
\end{proposition}

\begin{proof}
By Corollary~\ref{Yloop}, we may identify $gr^LY(\gl_{M|N})=gr^LY_{\mu}$ with the loop superalgebra $U(\gl_{M|N}[t])$ via
\[
gr^L_{r-1}t_{ij}^{(r)}\longmapsto (-1)^{\pa{i}}E_{ij}t^{r-1}.
\]
We consider the following composition
\[
gr^L\widehat{Y}_{\mu}^-\otimes gr^L\widehat{Y}_{\mu}^0\otimes gr^L\widehat{Y}_{\mu}^+\twoheadrightarrow gr^L\widehat{Y}_{\mu} \xrightarrow{\Gamma} gr^LY_{\mu}\cong U(\gl_{M|N}[t]).
\]
Let $n_a:=\mu_1+\mu_2+\ldots+\mu_a$ for short. By Proposition~\ref{quasi}, the image of $\pa{E}_{a,b;i,j}^{(r)}$ (resp. $\pa{D}_{a;i,j}^{(r)}$, $\pa{F}_{b,a;i,j}^{(r)}$) under the above composition map is $(-1)^{\pa{n_a+i}}E_{n_a+i, n_b+j}t^{r-1}$ (resp. $(-1)^{\pa{n_a+i}}E_{n_a+i,n_a+j}t^{r-1}$, $(-1)^{\pa{n_b+i}}E_{n_b+i, n_a+j}t^{r-1}$ ). By the PBW theorem for $U(\gl_{M|N}[t])$, the set of all monomials in
\begin{align*}
&\qquad\big\lbrace gr_{r-1}^LD_{a;i,j}^{(r)} \,|\, 1\leq a\leq m+n, \; 1\leq i,j\leq\mu_a, \, r\geq 1 \big\rbrace \\
&\cup\big\lbrace gr_{r-1}^LE_{a,b;i,j}^{(r)} \,|\, 1\leq a<b\leq m+n, \; 1\leq i\leq\mu_a, 1\leq j\leq\mu_b, \, r\geq 1 \big\rbrace\\
&\cup\big\lbrace gr_{r-1}^LF_{b,a;i,j}^{(r)} \,|\, 1\leq a<b\leq m+n, \; 1\leq i\leq\mu_b, 1\leq j\leq\mu_a, \, r\geq 1 \big\rbrace
\end{align*}
taken in certain fixed order forms a basis for $gr^LY_{\mu}$ and hence Proposition~\ref{ind2} follows.
\end{proof}

Let $Y_\mu^0$, $Y_\mu^+$ and $Y_\mu^-$ denote the subalgebras of $Y_{\mu}$ generated by all the $D$'s, $E$'s and $F$'s, respectively.
Along the proofs of Proposition~\ref{ind1} and Proposition~\ref{ind2}, we have found the PBW bases for each of these algebras.

\begin{corollary}
\begin{enumerate}
\item[(1)] The set of monomials in $\{ D_{a;i,j}^{(r)}\}_{1\leq a\leq m+n, 1\leq i,j\leq \mu_a, r\geq 1}$ taken in certain fixed order forms a basis for $Y_\mu^0$.
\item[(2)] The set of monomials in $\{ E_{a,b;i,j}^{(r)}\}_{1\leq a<b\leq m+n, 1\leq i\leq\mu_a,1\leq j\leq\mu_b, r\geq 1}$ taken in certain fixed order forms a basis for $Y_\mu^+$.
\item[(3)] The set of monomials in $\{ F_{b,a;i,j}^{(r)}\}_{1\leq a<b\leq m+n, 1\leq i\leq \mu_b,1\leq i\leq\mu_a, r\geq 1}$ taken in certain fixed order forms a basis for $Y_\mu^-$.
\item[(4)] The set of monomials in the union of the elements listed in (1), (2) and (3) taken in certain fixed order forms a basis for $Y_{\mu}$.
\end{enumerate}
\end{corollary}

\subsection*{Acknowledgements}
Thanks to my PhD supervisor Weiqiang Wang for his patient guidance and helpful comments improving this article.



\begin{thebibliography}{BK1}

\bibitem[BK1]{BK1}
J. Brundan and A. Kleshchev,
Parabolic Presentations of the Yangian $Y(\gl_n)$,
{\em Comm. Math. Phys.}
{\bf 254} (2005), 191-220.

\bibitem[BK2]{BK2}
J. Brundan and A. Kleshchev,
Shifted Yangians and finite $W$-algebras,
{\em Advances Math.}
{\bf 200} (2006), 136-195.

\bibitem[CP]{CP}
V. Chari and A. Pressley,
{\em A guide to quantum groups},
Cambridge University Press, 1994.

\bibitem[D1]{D1}
V. Drinfeld,
Hopf algebras and the quantum Yang-Baxter equation,
{\em Soviet Math. Dokl.} {\bf 32} (1985), 254--258.

\bibitem[D2]{D2}
V. Drinfeld,
A new realization of Yangians and quantized affine algebras,
{\em Soviet Math. Dokl.} {\bf 36} (1988), 212--216.


\bibitem[FRT]{FRT}
L. Faddeev, N. Reshetikhin and L. Takhtadzhyan,
Quantization of Lie groups and Lie algebras,
{\em Leningrad Math. J.} {\bf 1} (1990), 193--225.


\bibitem[GKLLRT]{GKLLRT}
I. Gelfand, D. Krob, A. Lascoux, B. Leclerc, V. Retakh and J.-Y. Thibon,
Non-commutative symmetric functions,
{\em Advances Math.} {\bf 112} (1995), 218--348.

\bibitem[Go]{Go}
L. Gow,
Gauss Decomposition of the Yangian $Y(\gl_{m|n})$,
{\em Comm. Math. Phys.}
{\bf 276} (2007), 799-825.

\bibitem[GR]{GR}
I. Gelfand and V. Retakh,
Quasideterminants, I,
{\em Selecta Math.} {\bf 3} (1997), 517--546.


\bibitem[KRS]{KRS}
P. Kulish, N. Reshetikhin, and E. Sklyanin,
Yang-Baxter equation and representation theory,
{\em Lett. Math. Phys.}
{\bf 5} (1981), 393--403.


\bibitem[Na]{Na}
M. Nazarov,
Quantum Berezinian and the classical Capelli identity,
{\em Lett. Math. Phys.} {\bf 21} (1991), 123--131.

\bibitem[MNO]{MNO}
A. Molev, M. Nazarov and G. Olshanskii,
Yangians and classical Lie algebras,
{\em Russian Math. Surveys} {\bf 51} (1996), 205--282.

\bibitem[Mo]{Mo}
A. Molev,
{\em Yangians and classical Lie algebras},
Mathematical Surveys and Monographs, {\bf 143}
American Mathematical Society, Providence,
RI, 2007.

\bibitem[Ta]{Ta}
V. Tarasov,
Irreducible monodromy matrices foran R-matrix of the XYZ-model and lattice local quantum Hamiltonians,
{\em Theoret. and Math. Phys.} {\bf 63} (1985), 440--454.


\bibitem[TF]{TF}
L. Takhtadzhyan and L. Faddeev,
The quantum method of the inverse problem and the Heisenberg XYZ-model,
{\em Russian Math. Serveys} {\bf 34} (1979), 11--68.


\end{thebibliography}
\end{document}